\documentclass[a4paper,12pt,reqno]{amsart}
\usepackage[normalem]{ulem}

\usepackage{tikz}

\usepackage{graphicx}
\usepackage{subcaption}

\usetikzlibrary{shapes,positioning}
\usetikzlibrary{decorations.markings}
\usetikzlibrary{graphs}

\makeatletter
\@namedef{subjclassname@2020}{\textup{2020} Mathematics Subject Classification}
\makeatother

\usepackage{fullpage}
\usepackage{amssymb}
\usepackage{amsmath}
\usepackage{mathtools}
\usepackage{enumitem}
\usepackage{mathrsfs}
\usepackage[all]{xy}
\usepackage{mathtools}
\usepackage{hyperref}
\usepackage{url}
\usepackage{bbm}
\usepackage{longtable}
\usepackage{dsfont}
\usepackage{upgreek}

\usepackage{lmodern}
\usepackage[OT2, T1]{fontenc}

\DeclareSymbolFont{cyrletters}{OT2}{wncyr}{m}{n}
\usepackage[british]{babel}

\usepackage[margin=1in]{geometry}
\numberwithin{equation}{section} \numberwithin{figure}{section}

\DeclareMathOperator*{\Osum}{\sum{}^*}

\DeclareSymbolFont{cyrletters}{OT2}{wncyr}{m}{n}
\DeclareMathSymbol{\Sha}{\mathalpha}{cyrletters}{"58}
\DeclareMathSymbol{\Be}{\mathalpha}{cyrletters}{"42}

\newcommand{\x}{\mathbf{x}}

\newcommand\F{\mathbb{F}}
\renewcommand\P{\mathbb{P}}
\newcommand\Z{\mathbb{Z}}
\newcommand\N{\mathbb{N}}
\newcommand\Q{\mathbb{Q}}
\newcommand\R{\mathbb{R}}

\renewcommand{\mod}[1]{ \ \left(\textnormal{mod}\ #1\right)}
\newcommand{\md}[1]{  \left(\textnormal{mod}\ #1\right)}
 
\renewcommand{\l}{\left}

\renewcommand{\r}{\right}
\renewcommand{\b}{\mathbf} 

\renewcommand{\c}{\mathcal} 

\renewcommand{\leq}{\leqslant}
\renewcommand{\geq}{\geqslant}
\renewcommand{\#}{\sharp}

\newcommand{\p}{\mathfrak{p}}

\newtheorem{lemma}{Lemma}

\newtheorem{theorem}[lemma]{Theorem}

\theoremstyle{definition}

\newtheorem{definition}[lemma]{Definition}
\newtheorem{remark}[lemma]{Remark}

\newtheorem*{ack}{Acknowledgements}
\newtheorem*{notation}{Notation}

\numberwithin{lemma}{section}

\begin{document}

\vspace{-0,5cm}

\title
{$6$-torsion and integral points on quartic surfaces}
  
\begin{abstract} 
We prove   matching upper and lower bounds for   the average   of the  $6$-torsion  of class groups of quadratic fields.
Furthermore, we count the number of integer solutions on an affine quartic surface.
\end{abstract}

\author{S. Chan} 
\address{
ISTA
 \\Am Campus 1 \\
3400 Klosterneuburg \\
Austria}
\email{stephanie.chan@ist.ac.at}

\author{P. Koymans} 
\address{Institute for Theoretical Studies\\
ETH Z\"urich \\
8092\\
Switzerland}
\email{peter.koymans@eth-its.ethz.ch}

\author{C. Pagano} 
\address{
Department of Mathematics
\\ Concordia University 
 \\ Montreal 
 \\  H3G 1M8   \\  Canada}   
\email{carlo.pagano@concordia.ca}

\author{E. Sofos} 
\address{
Department of Mathematics\\
Glasgow University  
\\ G12~8QQ \\ UK}
\email{efthymios.sofos@glasgow.ac.uk}


   \vspace{-1cm}

\maketitle

\thispagestyle{empty}

\tableofcontents

\vspace{-1,3cm}

\section{Introduction}   
One of the main invariants of 
class groups of   quadratic fields $\mathbb Q(\sqrt D)$ is
the size $h_n(D)$ of their  $n$-torsion.
It has been investigated by several mathematicians:
By the work of Gauss~\cite{gauss}  in $1801$ the average of $h_2(D)$ for $D<0$
  is a constant multiple of $\log |D|$ when ordering 
the number fields by $-D$. Davenport and Heilbronn~\cite{MR491593} proved in $1971$ 
that $h_3(D)$ has a constant average, while, 
Fouvry and Kl\"uners~\cite{MR2276261, FK2}  in $2007$ showed that $h_4(D)$ is on average a constant multiple of $\log |D|$. 
 The influential work of Smith~\cite{smith}
in $2017$ established the complete distribution of  $h_{2^k}(D)$. There are no other values of $n$ for which the right order of magnitude 
is known. For general $n$, there is     work on bounds for $h_n(D)$ on average by Soundararajan~\cite{MR1766097}, Heath-Brown--Pierce~\cite{MR3692746}, 
Frei--Widmer~\cite{MR4238263} and Koymans--Thorner~\cite{KT}.

The Cohen--Lenstra conjectures~\cite{MR756082} predict that $h_n(D)$ is of constant average 
for  $ n$ odd and is  $\log |D|$ on average  for  $n$ even.
Let $\mathcal{D}^{+}(X)$ and $\mathcal{D}^{-}(X)$ be the set of respectively positive and negative fundamental discriminants with absolute value up to $X$. 
In this paper we establish the right order of magnitude   for the $6$-torsion:  
\begin{theorem}
\label{t6Torsion}
For all $X\geq 5 $ we have 
\[
X \log X \ll \sum_{\substack{D \in \mathcal{D}^+(X)}} h_6(D) \ll X \log X
\hspace{0.8cm} \mathrm{and} \hspace{0.8cm}
X \log X \ll \sum_{\substack{D \in \mathcal{D}^-(X)}} h_6(D) \ll X \log X
.\]\end{theorem} 
This marks the first time that  Nair--Tenenbaum
techniques are applied in arithmetic statistics, clarified in the subsequent remark:
\begin{remark}
[Idea of the proof of Theorem~\ref{t6Torsion}]
Using the Davenport--Heilbronn parametrisation
we turn the sum $\sum_D h_6(D)$ into an average of the   function $2^{\omega(m)}$ over the values $m$ assumed by  
  a   polynomial in $4$ variables, where the integer vectors lie  in a  subset of $\mathbb R^4$ with   spikes.
This average is a special instance of sums of the following form:
 \begin{equation}\label{eq:sums of wolke type}
\sum_{a\in \mathcal A} f(c_a) \chi(c_a),\end{equation} where
\begin{itemize}
\item 
$\mathcal A$ is a countable set, 
\item 
$\chi:\mathcal A \to [0,\infty)$ is any function of finite support,
\item 
 $c_a$ is an ``equidistributed'' sequence of positive integers,
\item 
$f$ is a non-negative arithmetic function being multiplicative or   more general.
\end{itemize}In our companion paper~\cite{companion} we prove upper bounds for such sums; here
we provide its applications.
\end{remark}
  \subsection{Applications to arithmetic statistics}
The following is a more general version of Theorem~\ref{t6Torsion} on mixed moments:
\begin{theorem}
\label{thm:mixedmoments}
Fix any $s>0$. Then for all $X\geq 5 $ we have 
\[
X(\log X)^{2^s-1} \ll \sum_{\substack{D \in \mathcal{D}^+(X)}} h_2(D)^s h_3(D) \ll X (\log X)^{2^s-1}
\]
and
\[
X (\log X)^{2^s-1} \ll \sum_{\substack{D \in \mathcal{D}^-(X)}} h_2(D)^s h_3(D)  \ll X (\log X)^{2^s-1}
,\] where the implied constant depends at most on $s$.
\end{theorem}
\begin{remark}[Independence]
Theorems~\ref{t6Torsion} and~\ref{thm:mixedmoments}
are  the first results establishing the right order of magnitude for $h_n$ when 
 $n$ has  more than one  prime factor.
 Since $h_6=h_2 h_3$, the underlying problems   are   related to independent behavior of $h_2$ and $h_3$.
One cannot exclude a priori  that $h_2(D)$ and $h_3(D)$ correlate in a way that $h_{3}(D)$ attains   very large values when  $h_{2}(D)$ is large.
\end{remark}

Davenport and Heilbronn~\cite{MR491593} proved that $h_3(D)$
has constant average when $D$ ranges in $\mathcal D^+(X)$.
We show that the $D$ responsible for this fact are those for which 
  $h_2(D)$  is  essentially   $(\log |D|)^{\log 2}$. For a real number  $\varepsilon>-1$ define 
\begin{equation}\label{def:meshuggah}\mathfrak c(\varepsilon):= -\frac{ \varepsilon}{   1+\varepsilon  } + \log (1+\varepsilon)\end{equation} 
and note that $\mathfrak c(\varepsilon)>0$.
\begin{theorem} \label{thm:i dont believe how good this is!!!!!!!!!} 
For every fixed constants  $\varepsilon_1 \in (0,1) $ and $ \varepsilon_2>0$
and all  $X,z_3,z_4\geq 1$  with  
$  (\log X)^{(1+\varepsilon_2) \log 2}\leq z_4$
 we have 
$$\sum_{\substack{D \in \mathcal{D}^+(X)\cup \mathcal{D}^-(X)\\h_2(D)\notin (z_3,z_4)}} \left(h_3(D)-1\right)
\ll X\left(
\left(\frac{z_3^{1 /\log 2}}{(\log X)^{(1-\varepsilon_1)}}\right)^{\log(1+\varepsilon_1)}
+\frac{1}{z_4^{\mathfrak{c}(\varepsilon_2) /\log 2}}
\right)
,$$ 
where the implied constant depends only on $\varepsilon_i$.
\end{theorem}

We will use the companion paper 
to give certain bounds for the frequency of atypical values of additive functions in Theorem~\ref{tDecay}.
This has certain algebraic applications that we describe now.
  
Malle's conjecture~\cite{MR2068887} regards the number of extensions $K/\mathbb Q$ 
with prefixed Galois group when ordered by their discriminant $\Delta_K$.
The case of the full symmetric group $S_n$
has attracted special attention; here, the largest $n$ for which asymptotics  are known
is $n=5$ due to Bhargava~\cite{MR2745272}; this was later extended and generalized by 
Shankar--Tsimerman~\cite{MR3264252} and Bhargava--Shankar--Wang~\cite{arXiv:1512.03035}.

We will prove that for the vast majority of  $S_5$-extensions, the cardinality of ramified primes can only lie in a specific interval.
This  was first studied by Lemke Oliver--Thorne~\cite{LOT}, who proved that the cardinality of ramified primes is distributed 
according to the Gaussian distribution of approximate centre    $\log \log  |\Delta_K|$ and   length $(\log \log  |\Delta_K|)^{1/2}$.
Our work complements this by proving that the cardinality can only lie outside the interval with probability that 
decays exponentially fast. 

\begin{theorem}
\label{thm:tailstailstails}
For every fixed constants  $\varepsilon_1 \in (0,1) $ and $ \varepsilon_2>0$
and all  $X,z_1,z_2\geq 1$  with  
$ (1+\varepsilon_2)\log\log X\leq z_2$
 we have
$$
\#\{K \textup{ quintic } S_5 : |\Delta_K |\leq X, \omega(\Delta_K) \geq z_2\}
\ll X\exp\left(- \mathfrak{c}(\varepsilon_2) z_2\right)
$$
and
\begin{align*}
\#\{K \textup{ quintic } S_5 :  |\Delta_K | & \leq X, \omega(\Delta_K) \leq z_1\}
\\ & \ll X \exp\left(-\log(1+\varepsilon_1)\left(\left(1-\varepsilon_1\right) \log\log X-z_1\right)\right)
,\end{align*} 
where the implied constants depend only on $\varepsilon_i$.
\end{theorem}
Since $\#\{K \textup{ quintic } S_5 : |\Delta_K |\leq X\}$ has order  $X$ due to Bhargava~\cite{MR2745272}, one sees from
Theorem~\ref{thm:tailstailstails}   that $\omega(\Delta_K)$ must typically lie in the interval $(z_1,z_2)$.
 
Malle's conjecture for cubic $S_3$ fields 
was first established by Davenport--Heilbronn~\cite{MR491593}. The 
 error term was later greatly improved by Bhargava--Shankar--Tsimerman~\cite{MR3090184}
and Taniguchi--Thorne~\cite{MR3121674}. Our next result shows that for $100\%$  of $S_3$ fields,
the number of ramified primes $\omega(\Delta_K)$ lies
in a prescribed interval, giving an analog of Theorem~\ref{thm:tailstailstails}.

\begin{theorem}
\label{thm:s3s3tailtail}
For every fixed constants  $\varepsilon_1 \in (0,1) $ and $ \varepsilon_2>0$
and all  $X,z_1,z_2\geq 1$  with 
 $ (1+\varepsilon_2)\log\log X\leq z_2$
 we have
$$
\#\{K \textup{ cubic } S_3 : |\Delta_K |\leq X, \omega(\Delta_K) \geq z_2\}|
\ll X\exp\left(- \mathfrak{c}(\varepsilon_2) z_2\right)
$$
and 
\begin{align*}\#\{K \textup{ cubic } S_3 : |\Delta_K |&\leq X, \omega(\Delta_K) \leq z_1\}
\\ & \ll X \exp\left(-\log(1+\varepsilon_1)\left(\left(1-\varepsilon_1\right) \log\log X-z_1\right)\right)
,\end{align*} 
 where the implied constants depend only on $\varepsilon_i$.
\end{theorem}

\subsection{Applications to Diophantine equations} We   count the number of integer solutions of certain Diophantine equations,
examples of which are the quartic affine surface $$x_1^2 x_2^2+x_3^2+x_4^2 =N$$ and the affine quartic threefold $x_1^2 x_2^2+x_3^2x_4^2+x_5^2 =N$. 
More generally, our work will cover  the equation  \begin{equation}\label{eq:freddy freeloader}(x_1\cdots x_k)^2+x_{k+1}^2+x_{k+2}^2=N,\end{equation} 
whose number of variables  is roughly half the degree of the equation.

For $N\in \N $ let  $$ L(1,\chi_{-N}) = \sum_{m=1}^{\infty}  \l(\frac{-N}{m}\r) \frac{1}{m} \hspace{1cm} \textrm{and} \hspace{1cm} 
\mathfrak b(N)=\prod_{p\mid N}\bigg(1+\l(\frac{-1}{p} \r)\frac{1}{p}\bigg)  .$$  \begin{theorem}\label{thm:well tempered clavier vol.1} 
Fix   $k\in \mathbb N$ and let  $N$ range through positive square-free integers $ 3 \md 8 $.
\begin{itemize} 
\item  The number of $\b x \in \Z^{k+2}$ satisfying~\eqref{eq:freddy freeloader}
 is $$\asymp \mathfrak b(N)^{k-1}  L(1,\chi_{-N})    N^{\frac{1}{2} } (\log N)^{k-1},$$ where the implied constant depends only on  $k$.
 \item  The number of $\b x \in \Z^{2k+1}$ satisfying $$ (x_1\cdots x_k)^2+ (x_{k+1}\cdots x_{2k})^2 +x_{2k+1}^2=N$$
 is $\asymp \mathfrak b(N)^{2(k-1)}  L(1,\chi_{-N})  N^{\frac{1}{2} } (\log N)^{2(k-1)}$, where the implied constant depends only on  $k$.
 \item  The number of $\b x \in \Z^{3k}$ satisfying  $$ (x_1\cdots x_k)^2+ (x_{k+1}\cdots x_{2k})^2 + (x_{2k+1}\cdots x_{3k})^2=N$$
is $\asymp \mathfrak b(N)^{3(k-1)}  L(1,\chi_{-N})    N^{\frac{1}{2} }(\log N)^{3(k-1)}$, where the implied constant depends only on  $k$.
\end{itemize} \end{theorem} 

The upper bound in the first bullet point in Theorem~\ref{thm:well tempered clavier vol.1} follows from earlier work of Henriot \cite[Theorem 3]{MR2911138}. 
All cases of Theorem~\ref{thm:well tempered clavier vol.1} are special cases of the more general Theorem~\ref{thm:3sqrsmultiplc}, which allows us to 
 put   general multiplicative weights on the integer solutions  $x_i$ of 
$$ x_1^2+ x_2^2+x_3^2=N.$$ Its proof is given in \S\ref{s:affinemanin} and is based on Theorem~\ref{tMain} and  deep estimates  of Duke~\cite{duke} for the Fourier coefficients of    cusp forms. It is worth mentioning that matching upper and lower bounds for the 
number of solutions of $$ \hspace{4.23cm} x_1^2+x_2^2+p^2=N, \ \ \ (x_1,x_2 \in \Z, p \ \textrm{prime}),$$  were given 
via the semi-linear sieve by Friedlander and Iwaniec~\cite[Theorem 14.5]{operacribro} on the assumption of the Generalized Riemann hypothesis 
and the Elliott--Halberstam conjecture.

 \begin{figure}[h]   \begin{subfigure}{0.45\textwidth}
    \centering    \includegraphics[width=1\linewidth]{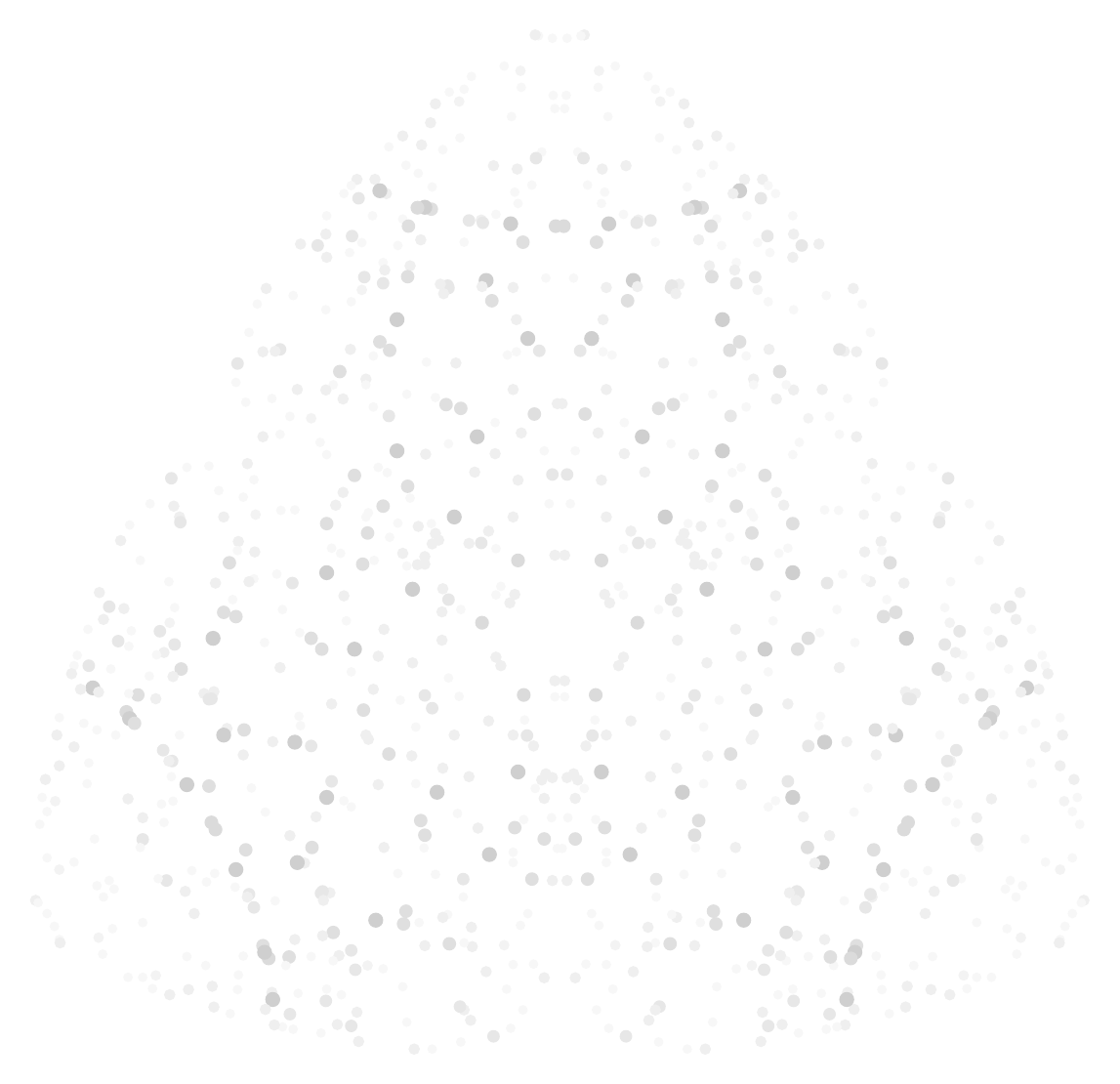}
    \label{fig:picture1}  \end{subfigure}    \begin{subfigure}{0.45\textwidth}
    \centering    \includegraphics[width=1\linewidth]{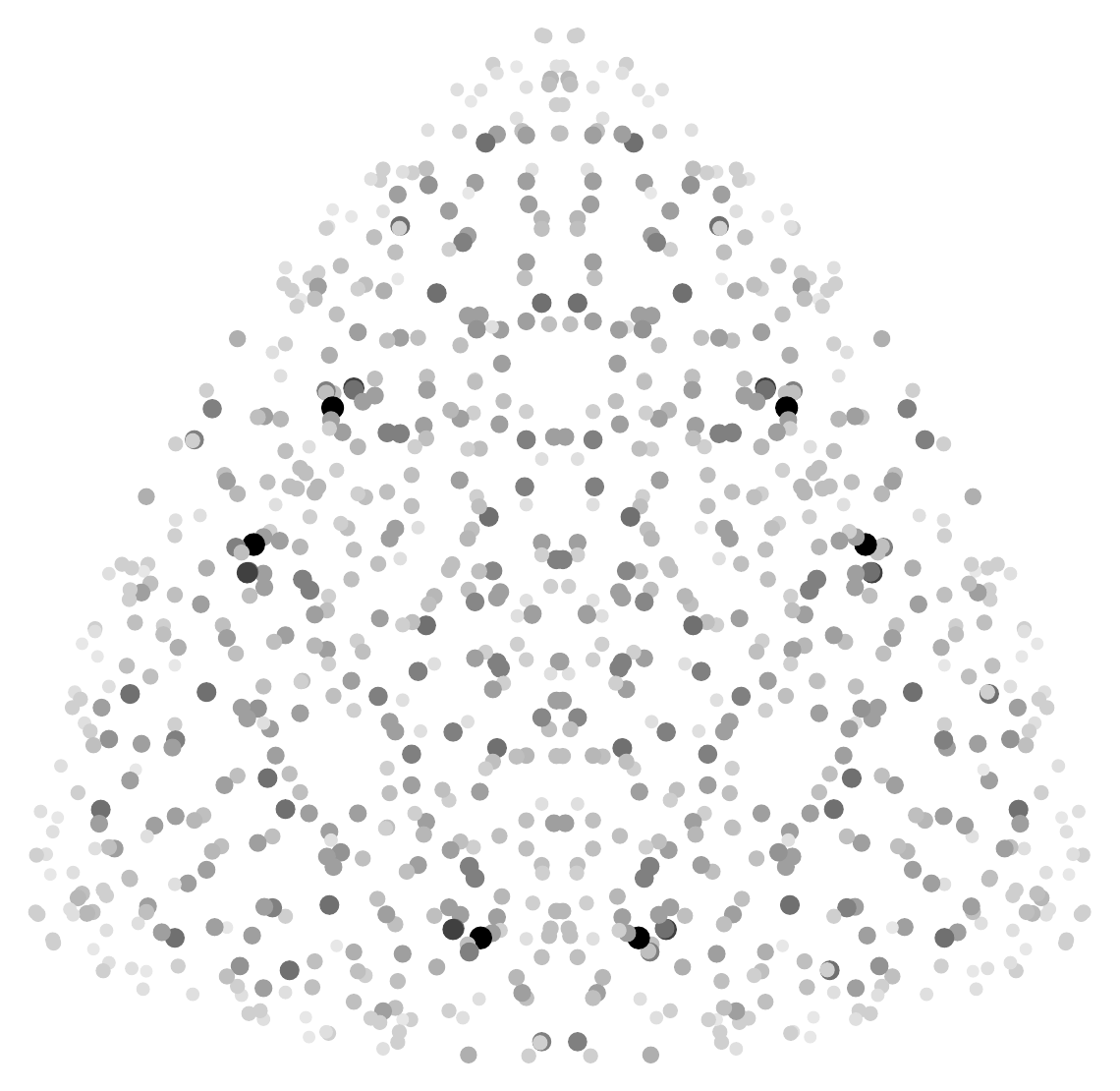}
    \label{fig:picture2}  \end{subfigure}  \vspace{-1em} \caption{Weighted  points on the sphere for $N=1716099$ and $N=1707035$}
    \label{fig11}
\end{figure}
 \begin{remark} [Bias]\label{rem:bias} 
The term $L(1,\chi_{-N})   N^{1/2} $ corresponds to the number of terms in the sum by a classical result of Gauss, 
whereas,     $(\log N)^{k-1}$  is   the average of  the $k$-th divisor function.
The shape of  $\mathfrak b(N)$ is  biased towards integers $N$ having many prime divisors $p\equiv 1\md 4 $ below $\log N$.
It is possible to combine this with  the work of Granville--Soundararajan~\cite[Theorem~5b]{MR2024414}, 
 to  find infinitely many $N$ such that $$ \#\{ \b x \in \N^6:   (x_1x_2)^2 + (x_3 x_4)^2 +(x_5 x_6)^2 =N  \} 
\gg   (\log \log N)^{5/2}   (\log N)^3 N^{1/2} .$$ This is in constrast with the typical size, which is  $(\log N)^3 N^{1/2}$
because   $L(1,\chi_{-N})$ and $ \mathfrak b(N)$ possess a limiting distribution
due to    the work of Chowla--Erd\H os~\cite{chowlaerdos} and Erd\H os--Wintner~\cite{erdwitner}.
  \end{remark}

The bias is illustrated in the   three-dimensional plots in Figure~\ref{fig11}.  
They depict points $\b x \in \N^3$ with $\sum_{i=1}^3 x_i^2=N$, where  each $\b x $   is colored   
based on the magnitude of  $\prod_{i=1}^3 \tau(x_i)$. 
The equations respectively have $960$ and $936$ solutions in        $\mathbb{N}^3$.
 Among the six primes that divide $1716099$, only one is  $1 \md 4$. However, in the factorization of $1707035$, 
 four primes are involved, and all except one are     $1 \md 4$.

 The ideas behind the proof of Theorem~\ref{thm:well tempered clavier vol.1}  are not specific to sums of three squares. We generalise the results to
equations without   specific shape, with the only provision that they have enough variables compared to the degree.
The end result is to study multiplicative functions over the coordinates of integer solutions of these Diophantine equations.
Problems of this type have been considered by Cook--Magyar~\cite{cookmag} and Yamagishi~\cite{yamagishi} in the case of the von Mangoldt function.

 \begin{theorem} \label{thm:spitfire}
Fix any $s>0$ and assume that $f:\N \to [0,\infty)$ is a multiplicative function satisfying 
$\tau(m)^{-s}\leq f(m) \leq \tau(m)^s$ for all $m$, where $\tau$ is the divisor function.
Assume that $F\in \Z[x_0,\ldots, x_n]$ is a smooth homogeneous polynomial of degree $d$ 
with   $n \geq 4+(d-1)2^d$ such that $F=0$ has a non-zero integer solution. Then ones has for all $B\geq 1 $
$$ \sum_{\substack{ \b x \in (\Z\setminus\{0\})^{n+1}\\  \max |x_i|\leq B, F(\b x)=0  }}
f(|x_0\cdots x_n |) \asymp B^{n+1-d}   \left\{  \exp \Big((n+1)  \sum_{p\leq B } \frac{(f(p)-1)}{p}\Big)
\right\},$$ where  the implied constants depend at most  on $F$ and $s$.\end{theorem}The proof is based on Birch's circle method~\cite{MR150129}.
The term $B^{n+1-d}$ represents the    number of terms in the sum over $\b x $.
It will be clear from the proof that the assumption $ f(m) \geq \tau(m)^{-s}$ is only needed for the lower bound.

\subsection{Polynomial values}\label{s:nrtnbmintroduction}
Nair and Tenenbaum proved upper bounds for the average of arithmetic functions evaluated over values of polynomials in~\cite{MR1618321}.
Such sums are omnipresent in number theory: their work was  crucial in many different problems. Examples include
\begin{itemize}
\item Equidistribution of {CM} points  (\cite{MR3898709}),
\item Manin's conjecture for counting rational points on surfaces (\cite{MR3953118},~\cite{MR3103132}), 
\item Mass equidistribution (\cite{MR2680498}),
\item Unit fractions  (\cite{MR3101397}). 
\end{itemize}
Such sums are of type~\eqref{eq:sums of wolke type} as can be seen
by taking $\chi$ to be the indicator function of an interval and $c_a$ to be the value of the polynomial at   an integer $a$. 
However,~\eqref{eq:sums of wolke type} also covers any polynomial in any number of variables.
In~\S\ref{s:nrtnbm} we shall   prove
Theorems~\ref{thm:nrtnbm} and~\ref{thm:lowernair}. These results respectively give matching upper and lower bounds for \begin{equation}\label{eq:sumsrrr}
\sum_{\substack{\mathbf x \in \Z^n\cap \mathcal R \\ Q(\mathbf x) \neq 0 }} f(|Q(\mathbf x)|)
,\end{equation} where  $\mathcal R$ is a bounded subset of $\mathbb R^n$ and $Q(\b x )$ is an arbitrary polynomial in any number of variables.
This is straightforward for polynomials without too many singularities but when $Q$ is very singular there is high probability that a small power of a prime divides 
its values. The new ingredient needed is an estimate by Pierce--Schindler--Wood~\cite[Lemma~4.10]{MR3551849} giving elementary proofs to 
statements regarding the Igusa zeta function.

We state one of the corollaries first. 
For $k,m\in \mathbb N$ denote the number of representations of $m$ as the product of $k$ natural numbers by $\tau_k(m)$. 
\begin{theorem}
\label{cor:nrtnbmadivisor} Let $ k \geq  2,  n\geq 1 $ be arbitrary integers, let $\ell$ be any positive real number
 and let $Q$ be an integer irreducible polynomial in $n$ variables.
Then for all $X\geq 2 $ we have
$$
X^n (\log X)^{k^\ell -1 }\ll
\sum_{\substack{ \mathbf x \in  (\mathbb Z \cap [-X,X])^n  \\ Q(\b x) \neq 0   }} \tau_k(|Q(\b x) |)^\ell
\ll  X^n (\log X)^{k^\ell -1 }
.$$  \end{theorem} 
\begin{remark}[Previously known cases]
Erd\H os~\cite{MR0044565} dealt with the case $\ell=n=1, k=2$ and 
Linnik~\cite{MR168543} with $n=1$,   all  $\ell,k\geq 1$ and  $\deg(Q)=1$. Later, 
Delmer~\cite{MR289433} worked in the cases 
when   $\ell\geq 1 , k=2 ,Q$ is irreducible and  $n=1$,
Nair--Tenenbaum~\cite{MR1618321} for all $\ell, k , Q$ and when $n=1$,
and de la Bret\`eche--Browning~\cite{MR2276196} whenever   $n = 2, \ell\geq 1 , k\geq 1 $ and   $Q$ is homogeneous.
Asymptotics for the divisor function over the values of polynomials in more
than one variable have been achieved by various authors, see, for example the work of  
de la Bret\`eche--Browning~\cite{MR2719554},
Zhou--Ding~\cite{MR4496536} and the list of references therein.
It is worth mentioning that in the case  $k=2,\ell=1$ and $Q$ a single irreducible polynomial in one variable,
only the cases corresponding to linear and quadratic polynomials are known to satisfy an asymptotic, see the work of 
Hooley~\cite{MR153648} and Bykovski\u{\i}~\cite{MR741852}.
\end{remark}
Theorem~\ref{cor:nrtnbmadivisor} 
follows directly from our next two results. Denote 
 $$\varrho_Q(q):= \frac{\#\{\b y \in (\Z/q\Z)^n: Q(\b y )\equiv 0 \md q\}}{q^n } 
$$ for any $Q\in \mathbb Z[x_1,\ldots, x_n]$ and $q\in \N$.

\begin{definition}\label{def:se pente lepta}
Let $\mathcal D\subset \R^n$ be   bounded. We let  $$X(\mathcal D):=\sup\left \{\max_{1\leq i \leq n }|x_i| :\b x \in \mathcal D\right \}.$$ 
\end{definition}

\begin{definition}[A class of functions]
\label{dMultRegion} Fix $ A\geq 1,\epsilon > 0, C>0$. The set $\mathcal M(A,\epsilon , C)$ of   functions
$f: \mathbb{N} \rightarrow [0,\infty)$ is defined by the property that for all coprime $m , n$ one has $$f(mn) \leq  f(m) \min\{A^{\Omega(n)}, C n^\epsilon \}.$$ 
\end{definition}
\begin{theorem} 
\label{thm:nrtnbm}
Fix $A\geq 1 $ and let  $\mathcal D$ be a bounded set. 
Let $Q$ be an arbitrary non-constant integer polynomial in $n$ variables without repeated polynomial factors over $\mathbb Q$
and let    $f$ be a function such that for every $\epsilon>0$ there exists $C>0$ for which $f \in \mathcal M (A,\epsilon,C)$. Then
$$\sum_{\substack{ \mathbf x \in  \mathbb Z^n \cap  \mathcal D \\ Q(\b x) \neq 0   }} f(|Q(\b x) |) \ll 
X(\mathcal D)^n
(\log 2X(\mathcal D))^{-r}
\sum_{\substack{ a\leq X(\mathcal D) ^n}} f(a)   
\varrho_Q(a)
,$$ where $r$ denotes the number of distinct irreducible polynomial factors of $Q$ over $\Q$
and   the implied constant depends at most on $A, f, n$ and $Q$.
\end{theorem}

For the corresponding lower bound to hold it is necessary that 
$\mathcal D$ is not too small; this explains the condition on $\mathcal D$ in our next
result:

\begin{theorem}
\label{thm:lowernair}Keep the notation and assumptions of Theorem~\ref{thm:nrtnbm}.
Assume, in addition, that $\mathcal D$ contains an open sphere of radius at least $X\geq 1$ and that   $f:\N\to [0,\infty) $ is a multiplicative function 
such that $$
\textrm{ for \ each }\   T \geq 1 \ \textrm{ one \ has } \ \inf\{f(m): \Omega(m) \leq  T \} > 0 
.$$ Then there exists a positive constant $\theta_Q$ that depends on $Q$ such that 
  $$\sum_{\substack{ \mathbf x \in  \mathbb Z^n \cap  \mathcal D \\ Q(\b x) \neq 0   }} f(|Q(\b x) |)
\gg X^n  (\log 2X )^{-r}
\sum_{\substack{ a\leq X}} f(a)    \varrho_Q(a)  
,$$ 
where the implied constants depend at most on $A, f, n$ and $Q$.
\end{theorem}

\begin{notation}
For a non-zero integer $m$ define $$\Omega(m):=\sum_{p\mid m } v_p(m),$$ where   $v_p$ is the standard $p$-adic valuation. Define $P^+(m)$ and $P^-(m)$ respectively to be the largest and the smallest  prime factor of a positive integer $m$ and let  $P^+(1)=1$ and $ P^-(1)=+\infty$. For a real number $x$
we reserve the notation $[x]$ for the largest integer not exceeding $x$. Throughout the paper we use 
  the standard convention that empty products are set equal to $1$. 
Throughout the paper we shall also make use of the convention that when   iterated logarithm functions $\log t,\log \log t$, etc., are used, the real variable $t$  is assumed to be sufficiently  large   to make the iterated logarithm well-defined.
\end{notation}

\begin{ack}The work         started during the research stay of SC, PK and CP during the 
workshop \textit{Probl\`emes de densit\'e en Arithm\'etique} at CIRM Luminy in $2023$. 
We would like to thank  the organisers Samuele Anni, Peter Stevenhagen and Jan Vonk.
We thank Levent Alp\"oge, Robert J.\ Lemke Oliver and  Frank Thorne for useful discussions on \S\ref{s:sou sfirizw}.
PK gratefully acknowledges the support of Dr.\ Max R\"ossler, the Walter Haefner Foundation and the ETH Z\"urich Foundation.
Part of the work of SC was supported by the National Science Foundation under Grant No.~\texttt{DMS-1928930}, while the 
author was in residence at the MSRI in Spring 2023.\end{ack}

\subsection*{Structure of the paper}
In \S\ref{s:thank msri for not letting me go to the workshop that started the project}  we recall the necessary results from~\cite{companion}.
Sections~\ref{s:sou sfirizw}-\ref{ss:954airport} respectively contain the proofs of
Theorems~\ref{t6Torsion} and~\ref{thm:mixedmoments} on the $6$-rank. In \S\ref{ss:lateformyflight} we   prove Theorem~\ref{tDecay} that provides 
tail bounds for the probability of large values of   additive functions in the general setting of Theorem~\ref{tMain}.  This is then  applied 
in Sections~\ref{ss:passengergordon}-\ref{ss:zone5please} and~\ref{ss:allremainingpassengers} to prove 
Theorems~\ref{thm:i dont believe how good this is!!!!!!!!!}-\ref{thm:tailstailstails} and~\ref{thm:s3s3tailtail} respectively
on $h_3$, $S_5$ and $S_3$ extensions. Sections~\ref{s:affinemanin}-\ref{ss:finalproofthrmes123} contain the proof of Theorem~\ref{thm:3sqrsmultiplc} on 
sums of three squares; this is more general than Theorem~\ref{thm:well tempered clavier vol.1}. The proof of 
Theorem~\ref{thm:spitfire} on general Diophantine equations is located in \S\ref{s:aplntratinpts}.
Lastly, in~\S\ref{ss:gnrluper} and \S\ref{ss:gnrlowerrs} we prove respectively 
Theorems~\ref{thm:nrtnbm} and~\ref{thm:lowernair} on averages of arithmetic functions over values of arbitrary polynomials.
They are then applied in \S\ref{s:prflastcorol} to prove Theorem~\ref{cor:nrtnbmadivisor}.

\section{Prerequisite lemmas}\label{s:thank msri for not letting me go to the workshop that started the project}
In this section we recall the required bounds proved in~\cite{companion}.
  \begin{definition}
[Density functions]
\label{densfnct} 
Fix $\kappa, \lambda_1, \lambda_2,  B, K  > 0$. We define the set $\mathcal D(\kappa,  \lambda_1, \lambda_2, B, K)$ of multiplicative functions
$h: \mathbb{N} \rightarrow \mathbb{R}_{\geq 0}$ by the properties  
\begin{itemize}
 \item for all $B <w < z$ we have 
\begin{align}
\label{elowaverg}
\prod_{\substack{p \  \mathrm{ prime} \\ w \leq p < z}} (1 - h(p))^{-1} \leq \left(\frac{\log z}{\log w}\right)^\kappa \left(1 + \frac{K}{\log w}\right)
,\end{align}
\item for every prime $p>B$ and integers $e\geq 1 $ we have 
\begin{align}
\label{eLowPower}
h(p^e) \leq \frac{B}{p},
 \end{align} 
\item for every prime $p$ and $e\geq 1 $ we have 
 \begin{align}
\label{eHighPower}
h(p^e) \leq p^{-e\lambda_1 + \lambda_2}.
\end{align}
\end{itemize}
\end{definition} 
Let $\mathcal A$ be an infinite set and for each $T\geq 1$ let $\chi_T:\mathcal A\to [0,\infty)$ be any function for which 
\begin{equation} \label{eq:baz1}\{a\in \mathcal A: \chi_T(a)>0 \} \ \textrm{ is finite for every } T\geq 1 .\end{equation} 
We also assume that \begin{equation} \label{eq:baz2}\lim_{T\to +\infty} \sum_{a\in \mathcal A} \chi_T(a)=+\infty .\end{equation} 
Assume that we are given a sequence of strictly positive integers 
$(c_a)_{a\in \mathcal A}$ indexed by  $\mathcal A$ and  denoted by 
$$\mathfrak C:= \{ c_a: a \in \mathcal A\}.$$
We will be interested in estimating sums of the form $$\sum_{a\in \mathcal A} \chi_T(a) f ( c_a ),$$ where $f$ is an arithmetic function.

We will need the following  notion of  `equi-distribution' of the values of the integer sequence $c_a$ in arithmetic progressions. 
For a non-zero integer $d$ and any $T\geq 1$, let 
$$
C_d(T)=\sum_{\substack{ a\in \mathcal A \\ c_a \equiv 0 \md d }} \chi_T(a).
$$ 

\begin{definition}[Equidistributed sequences]
\label{def:levdistr} 
We say that $\mathfrak C$ is equidistributed if there exist
positive real numbers $ \theta, \xi, \kappa, \lambda_1, \lambda_2,  B, K  $ with $\max\{\theta, \xi\}<1$, 
a function $M: \mathbb{R}_{\geq 1} \rightarrow \mathbb{R}_{\geq 1}$
and  a function $h_T\in \c D ( \kappa, \lambda_1, \lambda_2,  B, K)$   
such that
\begin{align}
\label{eEquiProgression}
C_d(T) =
h_T(d)
 M(T) \Bigg\{ 1+O\Bigg(\prod_{\substack{ B<p\leq M(T) \\  p\nmid d  }} (1-
h_T(p)
 )^2\Bigg)\Bigg\} + O(M(T)^{1-\xi}) 
\end{align}
for every $T \geq 1$ and every $d\leq M(T)^\theta$, where the implied constants are independent of $d$ and $T$.
\end{definition}

It is worth emphasizing that in this definition the constants 
$ \theta, \xi, \kappa, \lambda_1, \lambda_2,  B, K  $ are all assumed to be  independent of $T$. For example, the bound $h_T(p^e)=O(1/p)$ 
in~\eqref{eLowPower} holds with an implied constant that is  independent of $e,p$ as well as  $T$.
From now on we shall write  $M$ for $M(T)$.
We are now ready to state the main   result in~\cite{companion}.
\begin{theorem}\label{tMain}
Let $\mathcal A$ be an infinite set and for each $T\geq 1$ define $\chi_T:\mathcal A\to [0,\infty)$ to be any function such that  both~\eqref{eq:baz1} and~\eqref{eq:baz2} hold.
Take a sequence of strictly positive integers $\mathfrak C=(c_a)_{a \in \mathcal{A}}$. Assume that $\mathfrak{C}$ is equidistributed with respect to some positive constants $\theta,\xi,\kappa,\lambda_1,\lambda_2,B,K$ and functions 
 $M(T)$  and $h_T\in \mathcal D(\kappa,  \lambda_1, \lambda_2, B, K)$   
as in Definition~\ref{def:levdistr}.
Fix any $A>1$ and assume that  $f$ is a function 
such that for every $\epsilon>0$ there exists $C>0$ for which 
$f \in \mathcal M (A,\epsilon,C)$, which is introduced in Definition~\ref{dMultRegion}.
 Assume that there exists $\alpha > 0$ and $\widetilde{B}>0$ such that for all $T\geq 1 $ one has 
\begin{equation}\label{Roll Over Beethoven}
\sup\{c_a :a\in \mathcal A, \chi_T(a)>0 \} \leq \widetilde{B} M^\alpha,
\end{equation} where $M=M(T)$ is as in Definition~\ref{def:levdistr}.
Then for all $T\geq 1 $  we have
$$
\sum_{a\in \mathcal A } \chi_T(a) f(c_a) 
  \ll  M\prod_{\substack{ B<p\leq M    } } (1-
h_T(p)
) 
\sum_{\substack{ a\leq M  }} f(a)     h_T(a),
$$
where the implied constant is allowed to depend
on $  \alpha, A, B, \widetilde{B}, \theta  , \xi, K,  \kappa$, $\lambda_i$, the function $f$ and the implied constants in~\eqref{eEquiProgression}, 
but is independent of $T$ and $M$.
\end{theorem}Let us now recall the corresponding lower bound proved in~\cite{companion}. 
\begin{theorem}\label{thm:lower}Keep the notation and assumptions of Theorem~\ref{tMain}. 
Assume, in addition, that  $f:\N\to [0,\infty) $ is a multiplicative function for which 
$$\mathrm{ for \ each }\   L \geq 1 \ \mathrm{ one \ has } \ \inf\{f(m): \Omega(m) \leq  L  \} > 0 
$$ and that the  error term in Definition~\ref{def:levdistr}  satisfies 
$$ C_d(T) = h_T(d) M(T) \left\{ 1+o_{T\to\infty}\left(\prod_{\substack{ B<p\leq M(T) \\  p\nmid d  }} (1-h_T(p)
)^2\right)\right\} + O(M(T)^{1-\xi}) $$
whevever $d\leq M(T)^\theta$. Then for all $T\geq 1 $  we have $$ \sum_{a\in \mathcal A} \chi_T(a) f(c_a) 
\gg M(T)\prod_{  p\leq M(T)    } (1-h_T(p) ) \sum_{\substack{ a\leq M(T)  }} f(a)   h_T(a) ,$$ where the implied constants are
  independent of $T$ and $M$.\end{theorem} 
We finish this section with lemmas that will be needed in the forthcoming applications.
\begin{lemma}\label{lem:chebotrv} 
Let $Q\in \Z[x_1, \ldots,x_n]$ be non-constant and without repeated factors over $\Q$. Then as $x\to \infty $ one has 
$$ \sum_{p\leq x }  \frac{\#\{\b x \in \F_p^n: Q(\b x )=0\}}{p^{n-1}}=r  \frac{x}{\log x} (1+o(1) ),
$$  where $r$ is the number of distinct irreducible factors of $Q$ in $\Q[x_1, \ldots, x_n]$.
Furthermore, there exists a constant $c=c(Q)$ such that for $x\geq 2 $ one has 
$$ 
\prod_{p\ \mathrm{ prime}, p\leq x} 
\bigg(1-\frac{\#\{\b x \in \mathbb F_p^n: Q(\b x ) =0 \}}{p^n}\bigg)= c (\log x)^{-r} +O( (\log x)^{-r-1} ),
$$  where  the implied constant depends on  $Q$.
\end{lemma}

\begin{proof}
We factor $Q$ over $\overline{\Q}$ as $c_0  \prod_{i = 1}^t Q_i$ with $c_0$ in $\Q^{*}$, with $Q_i \in \overline{\Z}[x_1, \ldots,x_n]$ irreducible and with the property that if $Q_i$ occurs in the factorisation, then so does each of its Galois conjugates. We write $S = \{Q_1, \dots, Q_t\}$ for the set of factors obtained in this way. Let $K/\Q$ be the number field obtained by adding all the coefficients appearing in the factorisation. Since the factors come as Galois orbits, the field $K$ must be Galois. The group $\text{Gal}(K/\Q)$ acts on $S$ by permuting the factors. By the Lang--Weil bounds~\cite{MR65218} we have that
$$
\frac{\#\{\b x \in \mathbb F_p^n: Q(\b x ) =0 \}}{p^n} = \frac{c_Q(p)}{p}+O\left(\frac{1}{p^{\frac{3}{2}}}\right),
$$
where $c_Q(p)$ denotes the number of distinct irreducible factors of $Q$ defined over $\mathbb{F}_p$, when one factorizes the polynomial $Q \bmod p$ in $\overline{\mathbb{F}_p}[x_1, \ldots, x_n]$: in other words the irreducible factors in the $\mathbb{F}_p$-factorisation that remain irreducible factors in the $\overline{\mathbb{F}_p}$-factorisation. We now wish to express the function $c_Q(p)$ as a function of the Artin symbol of $p$ in $K/\Q$, for sufficiently large primes $p$.

For a prime $p$ that is also unramified in $K/\Q$, $\text{Art}(p, K/\Q)$ defines a conjugacy class in $\text{Gal}(K/\Q)$, which we view as permutations on $S$ via the action. The number of fixed points of the resulting permutation is independent of the element in the conjugacy class, and we denote this function of $\text{Gal}(K/\Q)$ as $g \mapsto \text{Fix}(g)$. This defines a function on sufficiently large primes via $p \mapsto \text{Fix}(\text{Art}(p,K/\Q))$. We claim that for $p$ sufficiently large we have that
$
c_Q(p)=\text{Fix}(\text{Art}(p,K/\Q)).
$
Indeed, observe that since $Q$ has no repeated factors over $\Q$, it follows that its reduction modulo $p$ has no repeated factors in $\mathbb{F}_p[x_1, \ldots, x_n]$ provided that we take $p$ sufficiently large. Furthermore, for $p$ sufficiently large, choosing any prime $\bar{\mathfrak{p}}$ above $p$ in $\overline{\Z}$, we have that all of the elements of $S$ remain irreducible when reduced modulo $\bar{\mathfrak{p}}$. 

We claim that if:
\begin{enumerate}
\item[(P1)] $Q$ has no repeated factors modulo $p$,
\item[(P2)] all of the elements of $S$ remain irreducible modulo $\bar{\mathfrak{p}}$,
\item[(P3)] $p$ is unramified in $K/\Q$,
\item[(P4)] $c_0$ is coprime to $p$, 
\end{enumerate} then $c_Q(p)=\text{Fix}(\text{Art}(p,K/\Q)).$
To see this, let us fix a prime $\mathfrak{p}$ of $\mathcal{O}_K$ lying above $p$. Recall that there is a unique element $\sigma \in \text{Art}(p,K/\Q)$ such that $\sigma(\mathfrak{p})=\mathfrak{p}$ and $\sigma(\alpha) \equiv \alpha^p \bmod \mathfrak{p}$ for each $\alpha$ in $\mathcal{O}_K$. Now let us reduce each $Q_i$ modulo $\mathfrak{p}$. The factors $Q_i$ remain distinct thanks to (P1) and also remain irreducible thanks to (P2). Since $p$ is unramified thanks to (P3), we can find the unique element $\sigma$ as above. Furthermore, $c_0$ being non-zero makes sure that $Q$ is not $0$ modulo $\mathfrak{p}$. Thus, the reduction of the $Q_i$ is truly the factorisation of $Q$ modulo $\mathfrak{p}$. 

If $\sigma(Q_i)=Q_i$, then all of the coefficients $\gamma$ of $Q_i$ satisfy $\gamma^p=\gamma$ when reduced modulo $\mathfrak{p}$, i.e. they are all in $\mathbb{F}_p$. So each fixed point of $\sigma$ gives an irreducible factor of $Q$ over $\overline{\mathbb{F}_p}$ that is already in $\mathbb{F}_p$. Conversely, suppose that $\sigma(Q_i) \neq Q_i$. Since the factors remain distinct modulo $\mathfrak{p}$ by (P1), then $\sigma(Q_i)$ and $Q_i$ are also distinct factors modulo $\mathfrak{p}$. But this means that the polynomial $Q_i$ modulo $\mathfrak{p}$ and the same polynomial with all coefficients raised to the power $p$ are different factors of $Q$ modulo $\mathfrak{p}$. In other words, $Q_i$ is not defined over $\mathbb{F}_p$. Hence we have precisely proved that under the assumptions (P1)-(P4), the quantity $\text{Fix}(\text{Art}(p,K/\Q))$ equals the number of $\overline{\mathbb{F}_p}$-irreducible components of $Q$ that are defined over $\mathbb{F}_p$, i.e. it equals $c_Q(p)$.

Recall that each of (P1)-(P4) is satisfied for all sufficiently large primes.
By the Chebotarev density theorem we obtain the following for all $x\geq 2 $, 
$$ \sum_{p\leq x }  \frac{\#\{\b x \in \F_p^n: Q(\b x )=0\}}{p^{n-1} }=\left(\frac{\sum_{g \in \text{Gal}(K/\Q)}\text{Fix}(g)}{\#\text{Gal}(K/\Q)} \right)
  \frac{x}{\log x} +O\left(\frac{x}{(\log x)^2 } \right)
.$$ Using partial summation we obtain a constant $B$ depending only on $Q$ such that 
\begin{equation}\label{ta gemista sto tapsi TWRA!!!}
\sum_{p \leq x} \frac{\#\{\b x \in \mathbb F_p^n: Q(\b x ) =0 \}}{p^n} = \left(\frac{\sum_{g \in \text{Gal}(K/\Q)}\text{Fix}(g)}{\#\text{Gal}(K/\Q)} \right) ( \log \log x )
+ B + O(1/\log x)
,\end{equation} from which one can deduce an asymptotic for the product over primes $p\leq x $ in the statement of the lemma by taking logarithms.

To complete the proof, it suffices to observe that if $G$ is a finite group acting on a finite set $X$ and if $\text{Fix}(g)$ denotes the number of fixed points of an element $g$ in $G$ viewed as permutation of $X$, then
$$\frac{\sum_{g \in G}\text{Fix}(g)}{\#G}
$$
equals the number of orbits of $G$ acting on $X$. In our case the number of $\text{Gal}(K/\Q)$-orbits acting on $S$ is $r$, thus completing the argument.
\end{proof}

\begin{lemma}
\label{formulasfataltotheflesh}  
Let $Q\in \Z[x_1, \ldots,x_n]$ be non-constant and without repeated factors over $\Q$. 
Then for any prime $p$ the number of $\b x \in (\Z/p^2\Z)^n$ for which $Q(\b x )\equiv 0 \md{p^2}$ 
is $O(p^{n-2})$, where the implied constant depends at most on $Q$.
  \end{lemma}\begin{proof}
For a point $\b t $ in $(\Z/p\Z)^n$ satisfying $Q(\b t)\equiv 0\md {p}$, we denote  
$$
N(\b t)=\#\left \{\b x  \in (\Z/p^2\Z)^n:Q(\b x )\equiv 0 \md {p^2}   \ \text{and} \ \b x \equiv \b t \md p \right \}.
 $$
By definition, we have that
$$
\#\left \{\b x  \in (\Z/p^2\Z)^n:Q(\b x )\equiv 0 \md {p^2}  \right \}
=
\sum_{\substack{ \b t \in (\Z/p\Z)^n \\ Q(\b t )\equiv 0 \md p }} N(\b t ). 
$$
Suppose that $\b t $ in $(\Z/p\Z)^n$ satisfies both $Q(\b t )\equiv 0\md{p}$ and $\nabla Q(\b t) \not \equiv \mathbf{0} \mod p$. 
Then Hensel's lemma implies that $N(\b t )=p^{n-1}$. 
The Lang--Weil estimates~\cite{MR65218} imply that the number of $\b t \in (\Z/p\Z)^n$ for which $Q(\b t )\equiv 0 \md p $ is $O(p^{n-1})$, hence, 
$$
\#\left \{\b x  \in (\Z/p^2\Z)^n:Q(\b x )\equiv 0 \md {p^2}  \right \}
=O(p^{2n-2})+
\sum_{\substack{ \b t \in (\Z/p\Z)^n \\ Q(\b t )\equiv 0 \md p \\ \nabla Q(\b t)  \equiv \mathbf{0} \md p }} N(\b t )
.$$ 
Since one has the trivial bound 
 $N(\b t )\leq \#\left \{\b x  \in (\Z/p^2\Z)^n: \b x \equiv \b t \md p \right \}=p^n$, it is sufficient for the proof to show that 
$$ 
\#\left \{\b t \in (\Z/p\Z)^n:Q(\b t )\equiv 0 \md p,  \nabla Q(\b t)  \equiv \mathbf{0} \md p \right\}=O(p^{n-2}).
$$
Since $Q$ has no repeated polynomial factors over  $\Q$, it follows that it 
has no repeated polynomial factors over  $\mathbb{F}_p$ for all sufficiently large primes $p$. For these $p$ we therefore see that 
the intersection $Q(\b t )=  \nabla Q(\b t)=\mathbf{0}$ defines a subvariety of $\mathbb{A}_{\mathbb{F}_p}^n$ of codimension at least $2$. 
The Lang--Weil estimates~\cite{MR65218}   therefore provide the required bound $O(p^{n-2})$.
\end{proof}

\begin{lemma}\label{lem:lord of all fevers and plague}
Fix any positive real numbers $c_0,c_1,c_2,c_3$ and assume that $F:\mathbb N\to [0,\infty)$ is a multiplicative function such that 
\begin{align}
\label{eFbound}
F(p^e) \leq \min \left\{\frac{c_0}{p}, \frac{p^{c_1}}{p^{ec_2}}\right\}
\end{align}
for all primes $p$ and $e\geq 1 $
and $F(p^e) \leq c_3/p^2$ for all $p,e\geq 2$.
Fix any $C,C'>0$ and 
 assume that 
$G:\mathbb N\to [0,\infty)$ is a function such   
 that for all coprime positive integers $a,b$ one has  $G(ab ) \leq  G(a) \min\{C^{\Omega(b)}, C'b^{c_2/2}\}$.

Then  for all $x\geq 1 $ we have 
$$ \sum_{\substack{n\leq x\\P^-(n)> c_0}} F(n) G(n) \ll  \exp\left(\sum_{c_0< p\leq x} F(p) G(p) \right)
,$$ where the implied constant depends at most on $c_i$ and $C,C'$.
\end{lemma}

\begin{proof} We define a    multiplicative function  
$H'$ such that when $p$ is  prime and $e\geq 2 $ one has 
$H'(p^e)= \min\{C^e, C'p^{c_2 e/2}\} $ 
while $H'(p)= G(p)$. It is not difficult to show that 
 for all coprime positive integers $a,b$  we have  $G(ab ) \leq  G(a) H'(b)$.
Hence, $G(b) \leq H'(b)$ for all $b$ and therefore the sum in the lemma is at most 
 $$ \sum_{\substack{n\leq x\\P^-(n)> c_0}} F(n) H'(n)  \leq 
\prod_{\substack{n\leq x\\P^-(n)> c_0}} \left(1+ \sum_{e\geq 1 } F(p^e) H'(p^e) \right)\leq \exp\l(\sum_{c_0<p\leq x, e\geq 1 } F(p^e) H'(p^e) \right)
 $$ due to the inequality $1+z\leq \mathrm e^z$ valid for all $z\in \mathbb R$.
Let $\mathfrak E$ be a positive integer that will be specified later. 
The contribution of  $e> \mathfrak E$ is at most 
$$p^{c_1} \sum_{e> \mathfrak E } p^{-e c_2}H'(p^e)
\leq 
C' p^{c_1} \sum_{e> \mathfrak E } p^{-e c_2/2} 
\leq C'  p^{c_1 - \mathfrak E c_2/2} (1-2^{- c_2/2} )^{-1}\ll p^{c_1 - \mathfrak E c_2/2}.
$$ Taking $  \mathfrak E$ to be the least positive integer satisfying $2(c_1 +2)/c_2 \leq  \mathfrak E  $ yields the bound $\ll p^{-2}$.
The contribution of the terms in the interval $[2, \mathfrak E]$ is 
$$\leq \sum_{2\leq e \leq \mathfrak E} F(p^e) H'(p^e)
\leq \sum_{2\leq e \leq \mathfrak E} F(p^e)  C^e
\leq \frac{c_3}{p^2} \sum_{2\leq e \leq \mathfrak E}   C^e \ll \frac{1}{p^2}\ll \frac{1}{p^2}.$$ 
Thus, the overall bound becomes 
$$
 \exp\l(\sum_{c_0<p\leq x, e\geq 1 } F(p^e) H'(p^e) \right)
\leq  \exp\l( \sum_{c_0<p\leq x  } F(p) H'(p) \right)\exp\l( \sum_{c_0<p\leq x  } O(1/p^2) \right)
,$$ which is sufficient because $H'(p)=G(p)$.
\end{proof}

\begin{lemma}\label{lem:wud}Assume that $Q\in \Z[x_1,\ldots, x_n]$ is non-constant. Then for all $e\geq 1 $ and primes $p$ we have 
$$p^{-en }\#\{\b x \in (\Z/p^e \Z)^n: Q(\b x ) \equiv 0 \md {p^e} \} \ll p^{-e/\deg(Q) },$$ where the implied constant depends at most on $Q$.
\end{lemma}

\begin{proof}
If $Q$ is homogeneous then the bound follows  from~\cite[Lemma~4.10]{MR3551849}. If not, then we can work with the homogenized version $Q_1$ of $Q$, which is a homogeneous polynomial in $n+1$ variables having the same degree satisfying $Q_1(\mathbf{x}, 1) = Q(\mathbf{x})$. Thus, using the homogeneity of $Q_1$, one has  
$$ 
\#\{\b x \in (\Z/p^e\Z)^n: Q(\b x )\equiv 0 \}=\frac{\#\{z \in (\Z/p^e\Z)^*, \b x \in (\Z/p^e\Z)^n: Q_1(\b x,z )\equiv 0 \}}{(p-1)p^{e-1} }.
$$ 
Applying~\cite[Lemma~4.10]{MR3551849} to $Q_1$ shows that the numerator in the right hand-side is $\ll p^{e (n+1)-e/\deg(Q_1)}$, which is sufficient.
\end{proof}

\begin{lemma}\label{lem:buchanan bus station}
Let $Q\in \Z[x_1, \ldots,x_n]$ be a non-constant polynomial having no   repeated factors over $\Q$.
Fix any $\lambda > 0, C \in [1,2)  $ and assume that $G:\N \to \mathbb R\cap [0,\infty)$ is multiplicative, that 
$G(p)=\lambda$ for every prime $p$,
that $G(p^e) \leq C^e$ for all $e\in \mathbb N$ and primes $p$ 
and that for all $\epsilon>0$ there exists $C'(\epsilon)>0$ such that $G(b)\leq C'(\epsilon ) b^\epsilon$ for all $b \in \mathbb N$.

Then there exists a positive constant $c$   that depends on $Q$ and $G$, 
such that when  $x\to \infty $ we have 
$$  \sum_{1\leq m \leq x}   G(m) \frac{\#\{\b x \in (\Z/m\Z)^n: Q(\b x )\equiv 0 \md{m} \}}{m^n} \sim c (\log x )^{\lambda r}
,$$   
where $r$ is the number of distinct irreducible factors of $Q$ in $\Q[x_1, \ldots, x_n]$.
\end{lemma}\begin{proof}
We employ Wirsing's result~\cite[Satz 1]{wirsing} with $$ f_0(m)= G(m) \frac{\#\{\b x \in (\Z/m\Z)^n: Q(\b x )\equiv 0 \md{m} \}}{m^{n-1}}.$$
By~\cite[Lemma~2.7]{MR3988669} we have $f_0(p^e) \leq G(p^e) \deg (Q) \leq C^e \deg(Q)$ if $Q$ is primitive, which implies a similar bound for non-primitive $Q$. This means that the assumption~\cite[Equation~(3)]{wirsing} is met.
To verify~\cite[Equation~(4)]{wirsing} we note that 
$$\sum_{p\leq x } G(p) \frac{\#\{\b x \in \F_p^n: Q(\b x )=0\}}{p^{n-1}}
=\lambda \sum_{p\leq x }  \frac{\#\{\b x \in \F_p^n: Q(\b x )=0\}}{p^{n-1}}
$$ is asymptotic to $\lambda r x/ \log x $ by Lemma~\ref{lem:chebotrv} and the  assumption that $G$ is constantly $\lambda$ on the primes.
Hence, as $x\to \infty$,~\cite[Equation~(5)]{wirsing} gives 
$$\sum_{m\leq x } 
f_0(m)\sim c' \frac{x}{\log x } \prod_{p\leq x } \left(1+\sum_{e\geq 1 } G(p^e)  \frac{\#\{\b x \in (\Z/p^e\Z)^n: Q(\b x )=0\}}{p^{en}}
\right)$$ for some positive constant $c'$. Finally, using  an argument that is  similar to the ones in the proof of 
Lemma~\ref{lem:lord of all fevers and plague} and making use of Lemmas~\ref{formulasfataltotheflesh} and~\ref{lem:wud} to control 
the contribution of terms with $e\geq 2 $
  shows that when $x\to \infty$ the last product over $p\leq x $ is asymptotic to 
$$ c'' \exp\left(\sum_{p\leq x } G(p)  \frac{\#\{\b x \in (\Z/p \Z)^n: Q(\b x )=0\}}{p^n} \right) 
$$ for some positive $c''$. Injecting~\eqref{ta gemista sto tapsi TWRA!!!}
shows that 
$$\sum_{m\leq x } f_0(m) \sim c''' x (\log x)^{\lambda r -1}$$ for some positive constant $c'''$. Noting that 
$$G(m) \frac{\#\{\b x \in (\Z/m\Z)^n: Q(\b x )\equiv 0 \md{m} \}}{m^n} =\frac{f_0(m)}{m} 
$$ and using partial summation concludes the proof.
\end{proof}
We shall later need the following  substitute of Wirsing's theorem for multiplicative functions for which 
the average over the primes is not known. 
 \begin{lemma}\label{monteverdi frammenti} 
Fix any $k\in \N $ and assume that   $f$ is a  multiplicative function  satisfying  $0\leq f(p^e) \leq \tau(p^e)^k p^{- e}$   for  all $e\geq 1 $ and primes $p$.
Then  for all $x\geq 2$ we have    $$   \sum_{n \leq x  } f(n)  
\asymp \exp\l(\sum_{p\leq x } f(p)\r),
$$  where the    implied constants depend at most on $k$.
\end{lemma}\begin{proof} The upper bound is evident. For the lower bound our plan is to prove that 
there exists $\delta=\delta(k) \in (0,1)$ such that 
\begin{equation}
\label{eq:keelhaul subject to change}   \exp\l(\sum_{p\leq x^\delta } f(p)\r)
\ll \sum_{n \leq x  } f(n)   .\end{equation}
This is clearly sufficient since 
$$ \sum_{x^\delta< p\leq x} f(p) \ll_k  \sum_{x^\delta <p\leq x}  \frac{1}{p} \ll_k 1.  $$
To prove~\eqref{eq:keelhaul subject to change} we start by noting  
 that for each  $y\in [2,x]$ one has 
$$ \sum_{n \leq x  } f(n)   \geq \sum_{\substack{ n \leq x\\P^+(n) \leq y }} f(n) \mu(n)^2
=  \sum_{ P^+(n) \leq y } f(n)\mu(n)^2  -  \sum_{\substack{ n > x\\P^+(n) \leq y }} f(n)  \mu(n)^2
.$$ Since  there exists  $C(k)>0$ such that $$  \sum_{ P^+(n) \leq y } f(n)\mu(n)^2  \geq   C(k) \exp\l(\sum_{p\leq y } f(p)\r),$$ 
it suffices to show that  $$  \sum_{\substack{ n > x\\P^+(n) \leq y }} f(n)\mu(n)^2 \leq \frac{C(k)}{2} 
 \exp\l(\sum_{p\leq y } f(p)\r).$$ 
We will see that this holds when $y=x^\delta$, where  $\delta$  is    a small positive constant that depends on $k$.
Define  $\sigma=1/\log y $ so that by Rankin's trick we have \begin{align*}
& \sum_{\substack{ n > x\\P^+(n) \leq y }} f(n)\mu(n)^2 \leq  x^{-\sigma} \sum_{P^+(n) \leq y } f(n) \mu(n)^2n^\sigma
\\=&  x^{-\sigma} \prod_{p\leq y } (1+   f(p)p^{\sigma }) 
\leq  x^{-\sigma}  \exp\l( \sum_{p\leq y} f(p)p^{\sigma }  \r) . \end{align*} 
Since $\mathrm e^t \leq 1 + \mathrm e \cdot t $ for $0\leq t \leq 1 $, we see that   
$p^\sigma \leq 1 + \mathrm e \cdot \sigma \log p$ for all primes $p\leq y$, hence, the sum inside the exponential is at most 
$$
  \sum_{p\leq y }f(p)  +O\l(  \sigma \sum_{p\leq y }  f(p) \log p \r)
\leq 
  \sum_{p\leq y }f(p)  +O\l(  \sigma \sum_{p\leq y }  \frac{ \log p }{p}  \r)
=  \sum_{p\leq y }f(p)  +O(1).$$
Hence, there exists a positive constant $C_1(k)$ such that  $$ \sum_{\substack{ n > x\\P^+(n) \leq y }} f(n)\mu(n)^2 \leq 
C_1(k) x^{-\sigma } \exp\l(\sum_{p\leq y } f(p)\r).$$ 
Denote $C_2(k)= C(k)/(2C_1(k) )$.
We   want to make sure that $x^{-\sigma }\leq C_2(k)$;
this can be achieved by taking $y=x^\delta$ with $\delta =\max \{1/2,(-\log C_2(k))^{-1} \}$. 
This is because we have   $x^\sigma=\mathrm e^{1/\delta}$ due to $y=x^\delta$.
\end{proof}

\section{Arithmetic statistics}
\subsection{$6$-torsion}\label{s:sou sfirizw}Here we prove Theorems~\ref{t6Torsion} using  Theorem~\ref{tMain}. 
This has an assumption related to a level of distribution result. 
Similar results have been obtained by~\cite[Theorem~1.2]{BTT},~\cite[Section 6]{EPW} and~\cite[Theorem~2.1]{LOT}. 
Here we use the one by Belabas~\cite[Th\'{e}or\`{e}me~1.2]{Belabas}. 
Let $g_1$ be the multiplicative function defined as
\[
g_1(p^e) =
\begin{cases}
p/(p+1), & \textup{if } p \geq 2 \textup{ and } e = 1 \\
0, & \textup{if } p > 2 \textup{ and } e \geq 2 \\
4/3, & \textup{if } p = 2 \textup{ and } e = 2 \\
4/3, & \textup{if } p = 2 \textup{ and } e = 3 \\
0, & \textup{if } p = 2 \textup{ and } e \geq 4.
\end{cases}
\] It is not difficult to see that 
\begin{equation}\label{eq:notdifficult}
\prod_{p\leq X}\big(1-\frac{g_1(p)}{p}\big) \sum_{a\leq X} 2^{\omega(a)} \frac{g_1(a)}{a} 
\leq 
\prod_{p\leq X}\big(1-\frac{g_1(p)}{p}\big)  \sum_{a\leq X} 2^{\omega(a)} \frac{4}{3a} 
\ll \log X.
\end{equation}

\begin{lemma} 
\label{tBelabas} Fix any $\epsilon > 0$. Then for all $q\in \mathbb N$ and $X\geq 2 $ with      $q < X^{\frac{1}{15}-\epsilon}$ we have 
\[
\sum_{\substack{D \in \mathcal{D}^+(X) \\ q \mid D}} (h_3(D) - 1)
= \frac{1}{\pi^2}  \frac{g_1(q)}{q}  X + O\left(\frac{X}{q(\log  X)^2 (\log\log X)^{2 - \epsilon} } + X^{\frac{15}{16} + \epsilon} q^{-\frac{1}{16}}\right)
\]
and
\[
\sum_{\substack{D \in \mathcal{D}^-(X) \\ q \mid D}} (h_3(D) - 1)
= \frac{3}{\pi^2}  \frac{g_1(q)}{q}  X + O\left(\frac{X}{q(\log  X)^2 (\log\log X)^{2 - \epsilon} } + X^{\frac{15}{16} + \epsilon} q^{-\frac{1}{16}}\right),
\]
where   the implied constants are independent of $q$ and $X$.
\end{lemma}

\begin{proof}
This follows from~\cite[Th\'{e}or\`{e}me~1.2]{Belabas} and the remark immediately thereafter.
\end{proof}
We are now ready to begin the proof of 
Theorem~\ref{t6Torsion}. The lower bounds follow from  $h_6(D) \geq h_2(D)$ and genus theory. This idea for the lower bound was further exploited and investigated in \cite[Section 5]{FK3}. For the upper bounds, we use Theorem~\ref{tMain} with 
\[\mathcal{A}=\{D\text{ fundamental discriminant}\},\ \chi_T(D)=(h_3(D)-1)\mathds 1_{|D|\leq T} (D),\ c_D=|D|.\]
We let  $h(q) = g_1(q)/q$ and     $f$ be the multiplicative
 function $f(n) = 2^{\omega(n)}$. Lemma~\ref{tBelabas} shows that the level of distribution assumption
in Definition~\ref{def:levdistr} is satisfied with $$M(T)=\frac{4}{\pi^2}T, \ \ \theta=\frac{1}{30} \ \ \textrm{ and } \ \  \xi=\frac{1}{32}.$$

Let $h_n^+(D)$ be the size of the $n$-torsion subgroup of the narrow class group. We have $h_n(D) \leq h_n^+(D)$. Since $h_2^+(D) = 2^{\omega(D) - 1}$ and $h_6^+(D) = h_2^+(D) h_3^+(D)$, we obtain 
\begin{align*}
&\sum_{\substack{D \in \mathcal{D}^+(X)}} h_6^+(D) + \sum_{\substack{D \in \mathcal{D}^-(X)}} h_6^+(D) \\
\\= 
&\sum_{\substack{D \in \mathcal{D}^+(X)}} h_2^+(D) + \sum_{\substack{D \in \mathcal{D}^-(X)}} h_2^+(D) +  \sum_{\substack{D \in \mathcal{D}^+(X)\cup \mathcal{D}^-(X)}} \left(h_3(D)-1\right)h_2^+(D).
\end{align*}
The first two sums are readily estimated as $O(X \log X)$. For the final sum, the application of Theorem~\ref{tMain} and~\eqref{eq:notdifficult}
yields \[\sum_{\substack{D \in \mathcal{D}^+(X)\cup \mathcal{D}^-(X)}} \left(h_3(D)-1\right)h_2^+(D)=\sum_{D\in\mathcal{A}} \chi_X(D)f(c_D)\ll X\log X,\] 
as required.

\subsection{Proof of Theorem~\ref{thm:mixedmoments}}\label{ss:954airport}The proof is as that   of Theorem~\ref{t6Torsion} with the only difference being that 
the bound 
$$   \sum_{a\leq X} \frac{ 2^{s\omega(a)} }{ a} 
\ll   \prod_{p\leq X}\big(1+\frac{2^s}{p}\big)
\ll(\log X)^{2^s}
$$
must be used in place of~\eqref{eq:notdifficult}. 
\subsection{Tail bounds for additive functions}\label{ss:lateformyflight}
We next      study $S_5$ and $S_3$-extensions for which it will be necessary to turn to the distribution of additive functions. Perhaps the most famous additive function is $\omega(n)$, which is roughly speaking normally distributed with mean $\log \log n$ and standard deviation $\sqrt{\log \log n}$ by the Erd\H{o}s--Kac theorem. However, these type of Erd\H{o}s--Kac results do not give any tail bounds for the frequency of very large values of additive functions. 
We will prove strong tail bounds in the general setting of Theorem~\ref{tMain}.
Recall the definition of $\mathfrak{c}(\epsilon)$ from~\eqref{def:meshuggah}.
\begin{theorem}\label{tDecay}
Let $\mathcal A$ be a set and  $\chi_T:\mathcal A\to [0,\infty)$ be a function for which both~\eqref{eq:baz1} and~\eqref{eq:baz2} hold. 
Fix any positive constants $\kappa,  \lambda_1, \lambda_2, B,  K,C,r$ and let $h \in \mathcal D(\kappa,  \lambda_1, \lambda_2, B, K)$  be such that \begin{equation}\label{eq:addassump1}
\sum_{\substack{ p\leq x  }}    h(p)  =r\log\log x+O(1) 
\end{equation}
 and such that for all primes $p$ and integers $e\geq 2 $ we have $h(p^e) \leq C/p^2$.
Fix any $\theta, \xi$ and let $\mathfrak{C}=(c_a)_{a \in \mathcal{A}}$ be an equidistributed sequence 
 as in Definition~\ref{def:levdistr}.
 Assume that there exists $\alpha > 0$ and $\widetilde{B}>0$ such that for all $T\geq 1 $ one has~\eqref{Roll Over Beethoven} and $M=M(T)$ is as in Definition~\ref{def:levdistr}.
Suppose $\psi: \mathbb{N} \rightarrow \mathbb{R}$ is an additive function such that there exists some constant $\tilde{A} \geq 1$ such that
\[
\psi(n) \leq \tilde{A} \Omega(n),\]
and that for every $\epsilon>0$ there exists $\tilde{C}\geq 0$ such that $\psi(n)\leq \epsilon\log m+\tilde{C}$.

Then for every fixed constants  $\varepsilon_1 \in (0,1) $ and $ \varepsilon_2>0$
and all  $T,z_2\geq 1$  with 
$$ 
r\tilde{A}(1+\varepsilon_2)\log\log M\leq z_2,$$
 we have
\begin{equation}\label{eq:tailhammer}
\sum_{\substack{a\in\mathcal{A}\\ \psi(c_a)\geq z_2}}\chi_T(a)\ll M\exp\left(-\frac{\mathfrak{c}(\varepsilon_2)}{\tilde A} z_2\right).
\end{equation} 
Further assume that there exists some constant $\tilde{A}_1>0$ such that
\begin{equation}\label{eq:lowerassume}
\sum_{p\leq x}(1-\varepsilon_1)^{\psi(p)/\tilde{A}_1}h(p)\leq r(1-\varepsilon_1)\log\log x+O(1),\end{equation}
then for all $T,z_1\geq 1$
\begin{equation}\label{eq:tailhammer2}
\sum_{\substack{a\in\mathcal{A}\\ \psi(c_a)\leq z_1}}
\chi_T(a)\ll M\exp\left(-\frac{\log(1+\varepsilon_1)}{\tilde{A}_1}\left(r\tilde{A}_1\left(1-\varepsilon_1\right) \log\log M-z_1\right)\right),
\end{equation} 
where   the implied constants depend  on $\varepsilon_i,\tilde{A}, C,\tilde{C}, \alpha, B, \widetilde{B}, \theta  , \xi, K,  \kappa,
\lambda_i$,$\psi$ and the implied constants in~\eqref{eEquiProgression}, but are independent of $z_i$, $T$ and $M$.
\end{theorem}
\begin{proof}
Fix a number $\beta>1$, which will be specified later. Define the multiplicative function $f_\beta(c)=\beta^{\psi(c)}$.
If $\psi(c)\geq z_2$ then $1\leq \beta^{-z_2} f_\beta(c)$, hence, 
 \[ \sum_{\substack{a\in\mathcal{A}\\ f_\beta(c_a)\geq z_2}}\chi_T(a) 
\leq
\beta^{-z_2}
 \sum_{a\in\mathcal{A}}\chi_T(a)f_\beta(c_a) 
.\]The assumptions on $\psi$ imply  that 
 $f_\beta \in\mathcal{M}\big(\beta^{\tilde{A}},\epsilon \log \beta, \beta^{\tilde{C}}\big)$ for every $\epsilon>0$. 
We can thus bound the sum in the right-hand side by Theorem~\ref{tMain}, hence, 
\begin{equation}\label{eq:afterMarkov}
\sum_{\substack{a\in\mathcal{A}\\ \psi(c_a)\geq z_2}}\chi_T(a)\ll
\frac{M}{\beta^{z_2}}\left(\prod_{\substack{ B<p\leq M    } } (1-h(p ) ) 
\sum_{\substack{ k\leq M  }} \beta^{\psi(k)}     h(k)  
\right),
\end{equation} 
where the implied constant is independent of $z_2$, $T$ and $M$.
We estimate the sum over $k$ in~\eqref{eq:afterMarkov} by applying Lemma~\ref{lem:lord of all fevers and plague} with
$F=h$ and $G=\beta^{\psi}$. This yields, by combining with~\eqref{eq:addassump1}, the upper bound
\[
\sum_{k \leq M} \beta^{\psi(k)} h(k)
\ll
\exp\left(\beta^{\tilde{A}} \sum_{p \leq M} h(p) \right)
\ll \left( \log M\right)^{r\beta^{\tilde{A}} }
.\] By~\eqref{eq:addassump1}, we also have
\begin{equation}\label{eq:denombd}
\prod_{B < p \leq  M} (1 - h(p)) \ll \exp\left(-\sum_{p < M} h(p)\right) \ll(\log M)^{-r}.
\end{equation}
This allows us to bound the right-hand side of~\eqref{eq:afterMarkov} by 
\[\ll
M \frac{(\log M)^{r(\beta^{\tilde{A}}-1)}}{\beta^{z_2}}
 \leq  M \exp \left (\left(\frac{\beta^{\tilde{A}}-1}{\tilde A (1+\varepsilon_2) }-\log \beta \right)z_2\right) 
 \] due to our assumption   $(\log M)^r  \leq \exp( z_2/(\tilde{A}(1+\varepsilon_2)) ) $.
Define $ \beta=(1+\varepsilon_2)^{1/\tilde A} $. Then 
$$
\frac{\beta^{\tilde{A}}-1}{\tilde A (1+\varepsilon_2) }-\log \beta
=
\frac{1}{\tilde A } \left(\frac{\varepsilon_2 }{  1+\varepsilon_2  }- \log (1+\varepsilon_2)   
 \right)
=-
\frac{\mathfrak{c}(\varepsilon_2) }{\tilde A }
.$$ This concludes the proof of~\eqref{eq:tailhammer}.

To prove~\eqref{eq:tailhammer2} we fix 
$\beta=(1-\varepsilon_1)^{1/\tilde{A}_1} \in (0,1)$ so that 
\begin{equation}\label{eq:conv}\frac{\beta^{\tilde{A}_1}-1}{\tilde{A}_1(1-\varepsilon_1)}-\log \beta=-\frac{\mathfrak{c}(-\varepsilon_1)}{\tilde{A}_1}<0.\end{equation}
If $\psi(c)\leq z_1$ then $\beta^{-z_1}f_\beta(c)\geq 1 $, hence, 
 \[ \sum_{\substack{a\in\mathcal{A}\\ f_\beta(c_a)\leq z_1}}\chi_T(a) 
\leq
\beta^{-z_1}
 \sum_{a\in\mathcal{A}_1}\chi_T(a)f_\beta(c_a)
.\]The assumptions on $\psi$ imply  that 
 $f_\beta \in\mathcal{M}(1,\epsilon, 1)$ for every $\epsilon>0$. We can thus bound the sum in the right-hand side by Theorem~\ref{tMain}, hence, 
\begin{equation}\label{eq:afterMarkov1}
\sum_{\substack{a\in\mathcal{A}\\ \psi(c_a)\leq z_1}}\chi_T(a)\ll
\frac{M}{\beta^{z_1}}\left(\prod_{\substack{ B<p\leq M    } } (1-h(p ) ) 
\sum_{\substack{ k\leq M  }} \beta^{\psi(k)}     h(k)  
\right),
\end{equation} 
where the implied constant is independent of $z_1$, $T$ and $M$.
We estimate the sum over $k$ in~\eqref{eq:afterMarkov1} by applying Lemma~\ref{lem:lord of all fevers and plague} with
$F=h$ and $G=\beta^{\psi}$. By~\eqref{eq:lowerassume}, this yields the upper bound
\[
\sum_{k \leq M} \beta^{\psi(k)} h(k)
\ll
\exp\left(\sum_{p \leq M} \beta^{\psi(p)} h(p)
 \right)
\ll \left( \log M\right)^{r\beta^{\tilde{A}_1} }
.\] 
Combining with~\eqref{eq:denombd}, this allows us to bound the right-hand side of~\eqref{eq:afterMarkov1} by 
\begin{align*}
&\ll
\frac{M}{\beta^{z_1}} (\log M)^{(\beta^{\tilde{A}_1}-1)r}
 \leq M\exp\left((\beta^{\tilde{A}_1}-1)r\log\log M-z_1\log \beta\right)\\
&\leq 
M\exp\left((\log \beta)\left(r\tilde{A}_1(1-\varepsilon_1)\log\log M-z_1\right)\right), 
\end{align*}
due to $-\log \beta<\frac{1-\beta^{\tilde{A}_1}}{\tilde{A}_1(1-\varepsilon_1)}$ from~\eqref{eq:conv}. 
Finally observe that
\[\log \beta=\frac{1}{\tilde{A}_1}\log(1-\varepsilon_1)<-\frac{1}{\tilde{A}_1}\log(1+\varepsilon_1)\]
 since $1/(1-\varepsilon_1)>1+\varepsilon_1$.
This concludes the proof of~\eqref{eq:tailhammer}.
\end{proof}
\subsection{Proof of Theorem~\ref{thm:i dont believe how good this is!!!!!!!!!}}
\label{ss:passengergordon}
Recall that $ 1/2\leq h_2(D)2^{-\omega(D)}\leq  1$. 
Hence, $h_2(D)\geq z_4$ implies  that $\omega(D)\geq \frac{\log z_4}{\log 2}$.
We   use Theorem~\ref{tDecay} with  $\psi=\omega, z_2= (\log z_4)/(\log 2)$ and \[\mathcal{A}=\{D\text{ fundamental discriminant}\},\ \chi_T(D)=(h_3(D)-1)
\mathds 1_{|D|\leq T} (D),\ c_D=D.\] As in the proof of Theorem~\ref{t6Torsion}, we pick $h$ to be the multiplicative function 
defined by $h(q) = g_1(q)/q$. We can check that~\eqref{eq:addassump1} is satisfied with $r=1$. Lemma~\ref{tBelabas} shows that the level of distribution 
assumption in Definition~\ref{def:levdistr} is satisfied with $M(T)=\frac{4}{\pi^2}T$, $\theta=\frac{1}{30}$ and $\xi=\frac{1}{32}$. The rest of the assumptions of 
Theorem~\ref{tDecay} can be readily verified. To bound the contribution of the cases with $h_2(D)\leq z_3$ we use~\eqref{eq:tailhammer2}  with  
$z_1=(\log z_3)/(\log 2)$.

\subsection{The proof of Theorem~\ref{thm:tailstailstails}} \label{ss:zone5please}
We use Theorem~\ref{tDecay} with 
\[\mathcal{A}=\{K \textup{ quintic } S_5\},\ \chi_T(K)=\mathds 1_{|\Delta_K|\leq T} (K),\ c_K=\Delta_K,\ \psi=\omega.\]
To show that $(c_K)$ is equidistributed, we use Bhargava's parametrization of $S_5$-extensions. The estimates from~\cite[Theorem~6]{MTT} implies that
\[C_d(T) = h(d)  \frac{13}{120} T + O(T^{\frac{399}{400}}) 
\]
for all $d\leq T^{\frac{3}{400}}$, where $h$ is multiplicative, and satisfies the properties that
 \[h(p)=\frac{(p+1)(p^2+p+1)}{p^4+p^3+2p^2+2p+1}=\frac{1}{p}+O\Big(\frac{1}{p^2}\Big)\]
for $p > 5$, and $h(p^e)\leq h(p^2)\ll 1/p^2 $ for all $e\geq 2$. Moreover $h(p^e)=0$ for all $e\geq 5$, $p > 5$, and also $h(p^e) = 0$ for $p \in \{2, 3, 5\}$ and $e$ sufficiently large ($e \geq 100$ suffices). Hence~\eqref{eEquiProgression} is satisfied with $M(T)=\frac{13}{120}T$, $\theta=\frac{3}{400}$ and $\xi=\frac{1}{400}$.
We let $\psi=\omega$ and $ r=\tilde A=\tilde A_1=1$. The assumption  $\psi(n) \leq \epsilon \log n+\tilde C$ is met due to 
the bound  $\omega(n) \ll (\log n) /(\log \log n )$ that is a consequence of the Prime Number Theorem.
An application of Theorem~\ref{tDecay} then concludes the proof.
\subsection{The proof of Theorem~\ref{thm:s3s3tailtail}}\label{ss:allremainingpassengers}
The arguments are similar to the 
ones in the proof of Theorem~\ref{thm:tailstailstails}. The only difference is that the uniformity estimates are   imported from 
the work of 
Bhargava--Taniguchi--Thorne~\cite[Theorem~1.3]{BTT}. Specifically, with the notation 
\[\mathcal{A}=\{K \textup{ cubic } S_3\},\ \chi_T(K)=\mathds 1_{|\Delta_K|\leq T} (K),\ c_K=\Delta_K\]
one has 
\[C_d(T) = h(d)  \frac{4}{12\zeta(3)}T + O(T^{\frac{5}{6}}) 
\]
for all $d\leq T^{\frac{1}{20}}$,   where the implied constant is absolute
and 
$h$ is multiplicative    satisfying
\[
h(p^e)=\begin{cases}
(p+1)/(p^2+p+1)&\text{if } e=1,\\
O( 1/p^2)&\text{if } e=2,\\
0&\text{if }e\geq 3
\end{cases}
\]
for every prime $p > 3$. We also have $h(2^e) = 0$ and $h(3^e) = 0$ for sufficiently large $e$. Since $h(p) = 1/p+O(1/p^2)$
we can employ Theorem~\ref{tDecay} with 
$\psi=\omega$ and 
$r=\tilde A=\tilde A_1=1$.

\section{Diophantine equations}
\subsection{Three squares} \label{s:affinemanin}
Denote $L(1,\chi_{-N})=\sum_{m=1}^\infty \l(\frac{-N}{m}\r)m^{-1}$, where $(\frac{\cdot}{\cdot})$ is the Legendre symbol.
A   theorem of Gauss states that for positive square-free  $N\equiv 3 \md 8$ one has 
$$  \#\{ \b x \in  \Z^3 :  x_1^2 + x_2^2 +x_3^2 =N \}=\frac{8}{\pi } 
L(1,\chi_{-N}) N^{1/2}.$$ 
The  main result of this section allows to put 
multiplicative weights on each variable. For $N\in \N$ and an arithmetic function $f$ we define 
\begin{equation}\label{def:obscura valediction}c_f(N):=\prod_{p\mid N} \bigg(1+\l(\frac{-1}{p} \r)\frac{(f(p)-1)}{p}\bigg).\end{equation}
 \begin{theorem}\label{thm:3sqrsmultiplc} Fix any $A>0, s>0, \alpha_1,\alpha_2,\alpha_3 \in  \{0,1\}$ and let $\alpha=\sum_{i=1}^3 \alpha_i$. 
Assume that  $f:\N\to[0,\infty)$ is a multiplicative function  such that $f(ab ) \leq  \tau(a)^s f(b)$ holds for all $a,b\in \N$.
Then for all positive square-free  integers  $N\equiv 3 \md 8 $   we have  
$$ \sum_{ \substack{ \b x \in  (\Z\setminus \{0\} )^3 \\ x_1^2 + x_2^2 +x_3^2 =N  }} \hspace{-0,2cm} 
\prod_{i=1}^3 f(|x_i|)^{\alpha_i} \hspace{-0,1cm} 
 \ll  L(1,\chi_{-N})  N^{1/2}  \l(  c_f(N)^{\alpha}    \exp\l(\alpha \sum_{p\leq N} \frac{f(p)-1}{p} \r) + \frac{1}{(\log N)^A} \r)
,$$  where    the implied constant is independent of $N$.

If, in addition, for  each $L \geq 1$ one   has $\inf\{f(m): \Omega(m) \leq  L  \} > 0 $, 
then for all positive square-free  integers  $N\equiv 3 \md 8 $   we have  
$$ \sum_{ \substack{ \b x \in  (\Z\setminus \{0\} )^3 \\ x_1^2 + x_2^2 +x_3^2 =N  }} \hspace{-0,2cm} 
\prod_{i=1}^3 f(|x_i|)^{\alpha_i} \hspace{-0,1cm} 
\gg  L(1,\chi_{-N})  N^{1/2}  \l(  c_f(N)^{\alpha}    \exp\l(\alpha \sum_{p\leq N} \frac{f(p)-1}{p} \r)  - \frac{1}{(\log N)^A} \r)
,$$  where    the implied constant is independent of $N$.
\end{theorem}

The  case $\alpha=1$ corresponds to imposing          weights on  one coefficient only; 
its proof is given in \S\ref{s:onevarr}. It is a straightforward combination of 
Theorem~\ref{tMain} and work of Duke~\cite{duke}, the details of which are recalled   in \S\ref{s:edstrbns}.
The proof in the cases with $\alpha=2,3$ require additional sieving arguments and for reasons of space we give
the full details only in the harder case $\alpha=3$. Specifically,  in \S\ref{s:trnvfmt}
we transform the sums into ones where  $\prod_{i=1}^3 f(|x_i|)$ is replaced by $f(\prod_{i=1}^3 |x_i| )$.
Subsequently, in \S\ref{ss:lebvelfdisnv78}
we  prove the requisite level of distribution for the transformed sums. 
Finally,  in \S\ref{ss:finalproofthrmes123} we prove Theorem~\ref{thm:3sqrsmultiplc}.
\subsection{Input from cusp forms}\label{s:edstrbns}
The main result in this subsection is Lemma~\ref{lemma:dukethm}; it regards
the number of solutions of $x_1^2+x_2^2+x_3^2=N$, with each $x_i$  divisible by an arbitrary integer $d_i$.
This is closely related to work of Br\"udern--Blomer~\cite[Lemma~2.2]{blombrud}
in the case where each $d_i$ is square-free. The proof of Lemma~\ref{lemma:dukethm}
 combines   the work of Duke~\cite{duke}
with that of  Jones~\cite{ednajones}.  We recall~\cite[Theorem~2, Equation~(3)]{duke}:
\begin{lemma}[Duke]\label{lem:duke}There exists a positive constant $\kappa$ such that for every positive definite quadratic integer ternary form $q$
and every square-free  integer $N$ one has 
$$
\#\{\b x \in \Z^3: q(\b x ) =N \}
= \kappa L(1,\chi_{q,N}) 
\mathfrak S(q,N)
\frac{\sqrt N}{\sqrt D}+O(D^{6} N^{1/2-1/30})
,$$ where the implied constant is absolute,
$D$ is the determinant of the matrix $(\partial ^2 q/\partial x_i \partial x_j)$, 
$\chi_{q,N}$ is the Dirichlet character 
$\chi_{q,N}(m )= (\frac{ -2 D \mathrm{disc}(\Q(\sqrt N))}{m })$ and 
$$\mathfrak S(q,N):= \prod_{p\mid 2D} \lim_{\lambda\to \infty }
\frac{\#\l\{\b x \in (\Z/p^\lambda \Z)^3: q(\b x ) \equiv N \md {p^\lambda}\r\}}{p^{2\lambda }}
.$$\end{lemma} Note that the definition of $\mathfrak S$ in~\cite[Equation~(4)]{duke} involves a finite value of $\lambda$, however, this is equivalent since these densities stabilise owing to the fact that $N$ is square-free.
We now  specify the constant  $\kappa$.
When $q=\sum_{i=1}^3 x_i^2$ we have $D=8$, hence,  
$$\#\{\b x \in \Z^3: x_1^2+x_2^2+x_3^2 =N \}
= 
\frac{ \kappa\c N(16^3) }{ 2 \sqrt 2 }\sqrt N
\sum_{m=1}^\infty \l(\frac{-4N}{m}\r)\frac{1}{m} 
+O(  N^{1/2-1/30}),$$where  
 $ \c N(m)= \#\{\b x \in (\Z/m\Z)^3:x_1^2+x_2^2+x_3^2\equiv N \md m \} m^{-2} $. 
 \begin{lemma}
\label{lem2dog} For any integer $N \equiv 3 \md 8 $ and $t\geq 3 $,  the  number of solutions of 
$x_1^2+x_2^2+x_3^2 \equiv N \md {2^t}$ is $4^t$.
\end{lemma} \begin{proof} Since $N\equiv 3 \md 8 $ every $x_i$ must be odd.
Let $x_1,x_2$ run through all odd elements $\md {2^t}$ and then        count the number of $x_3$ for which 
$x_3^2 \equiv a \md {2^t}$, where $a\equiv N-x_1^2-x_2^2 \md {2^t}$. 
Here $N\equiv 3 \md 8 $, hence, $a\equiv 1 \md 8 $.
Now we use the following   fact: for $t\geq 3 $ and 
 each $a \in \Z/2^t\Z$  with  $a\equiv 1 \md 8 $, the number of solutions of $x^2 \equiv a \md {2^t}$ is $4$.
This gives a total number of solutions $ 2^{t-1} \cdot 2 ^{t-1} \cdot 4 = 4^t$. 
\end{proof}
In particular, $ \c N(16^3)=1 $.
We obtain   $$  
\#\{\b x \in \Z^3: x_1^2+x_2^2+x_3^2 =N \}
= 
\frac{ \kappa }{ 2 \sqrt 2 }\sqrt N
\sum_{m=1}^\infty \l(\frac{-4N}{m}\r)\frac{1}{m} 
+O(  N^{1/2-1/30}).$$ 
By~\cite[Theorem~B, page 99]{bateman} this equals 
$ \frac{16}{\pi}\c L \sqrt N$, where $\c L:=\sum_{m=1}^\infty \l(\frac{-4N}{m}\r)\frac{1}{m}$. By 
Siegel's theorem we have $\c L \gg N^{-1/60}$, hence, 
$$\frac{ \kappa }{ 2 \sqrt 2 }- \frac{16}{\pi}  =O\l(\frac{1}{N^{1/30}\c L}\r)=O\l(\frac{1}{N^{1/60}}\r)
,$$ which shows that  $  \kappa=   32 \sqrt 2 /\pi $.

\begin{lemma}
\label{lemma:dukethm}For each $\b c \in \N^3$ and positive 
square-free $N\equiv 3 \md 8 $ we have 
$$ \#\l\{\b x \in \Z^3: \sum_{i=1}^3 (c_i x_i)^2=N \r\}
=\frac{8}{\pi } 
L(1,\chi_{-N})h_N(\b c ) N^{1/2}
 +O((c_1c_2c_3) ^{12} N^{1/2-1/30})
,$$ where the implied constant is absolute, $h_N(\b c )$ is given by $$
 \frac{2^{ \#\{ p\mid c_1 c_2c_3\}} }{c_1 c_2 c_3}
\prod_{\substack{ p\mid c_1 c_2 c_3   }} \left ( 1-\frac{(\frac{-   N }{p})}{p}\right) 
\mathds 1(\eqref{qweeebok0}-\eqref{qweeebok3})
\hspace{-0,5cm} 
\prod_{\substack{ p\mid c_1 c_2 c_3   \\ p \mathrm{ \ divides \ exactly \  one\  } c_i}}
\hspace{-0,5cm} 
2^{-\mathds 1 (p\nmid N ) }
\l(1     - \frac{1}{p} \l(\frac{-1}{p}\r)  \r)
 $$
and \begin{align}
 & 2\nmid c_1c_2c_3, \label{qweeebok0} 
\\ & \mathrm{gcd}(c_1,c_2,c_3)=1  , \label{qweeebok1} 
\\ & p \mathrm{ \ divides \ exactly \ two \  } c_i  \Rightarrow  \l(\frac{N}{p}\r)=1,\label{qweeebok3} 
\\ & p\mid N, p \mathrm{ \ divides \ exactly \ one \  } c_i  \Rightarrow  p\equiv 1 \md 4.\label{qweeebok2} 
\end{align}
\end{lemma}
\begin{proof}
We use  Lemma~\ref{lem:duke} with 
   $q= \sum_{i=1}^3 c_i^2 x_i^2$ so that $ D = 8 (c_1 c_2 c_3 )^2$.
Let us note that 
 $ \mathrm{disc}(\Q(\sqrt N))= 4 N,$ hence, the character $\chi_{q,N}(m)$ is given by 
 $$   \l(\frac{ -2 \cdot 8 (c_1 c_2 c_3 )^2 \cdot  4 N }{m }\r)
=\l(\frac{ -   N }{m }\r)\mathds 1 (\gcd(2c_1 c_2 c_3, m)=1).$$ Therefore, the  value of the corresponding $L$-function at $1$ is 
$$
 \prod_{p\nmid 2 c_1c_2c_3} \frac{1}{\left ( 1-\frac{(\frac{-   N }{p})}{p}\right) }
=
 \frac{ L(1,\chi_{-N})}{2} 
 \prod_{\substack{ p\mid c_1c_2c_3 \\ p\neq 2 }} \left ( 1-\frac{(\frac{-   N }{p})}{p}\right) 
.$$ 
To work out the   term $\mathfrak S$ we use 
 the work of Jones~\cite{ednajones}.
In the terminology of~\cite[Theorem~1.3]{ednajones} 
we take  $Q=\sum_{i=1}^3 (c_i x_i)^2, m=N$.
When $p\neq 2$ divides exactly one of the $c_i$, say, $c_3$,
then we take 
$a=c_1^2, b_1=0$ and~\cite[Equation~(1.5)]{ednajones}   shows that 
the $p$-adic factor in $\mathfrak S$ equals 
$$
\mathds 1 (p\mid N) 2\l(1-\frac{1}{p} \r) \mathds 1 (p\equiv 1 \md 4 )
+
\mathds 1 (p\nmid N)  \l(1-\frac{1}{p} \l(\frac{-1}{p} \r)\r) 
.$$
If $p$ divides exactly two of the $c_i$'s, say $c_2$ and $c_3$ then by taking 
$a=c_1^2$ in~\cite[Equation~(1.4)]{ednajones} shows that the $p$-adic factor in $\mathfrak S$ becomes 
 $2   $ or $0$, according to whether $(\frac{N}{p} )=1$ or not.
Finally, since $N$ is square-free, there is no prime $p$ that divides every $c_i$ since that would imply that $p^2 $ divides $N$.
Furthermore, by Lemma~\ref{lem2dog}  the $2$-adic density equals  $1$.
\end{proof}

\subsection{The one-variable case}\label{s:onevarr} The case $\alpha_1=1, \alpha_2=\alpha_3=0$ can be treated in a straightforward manner and 
we deal with it in this subsection. We  use Theorem~\ref{tMain} with $$\mathcal A=  (\Z\setminus \{0\} )^3,  c_a=|y_1|,
T=N, \chi_N(a)=\mathds 1 _{\{N\}} (y_1 ^2+ y_2 ^2+ y_3^2), M(N)=
\frac{8}{\pi } 
L(1,\chi_{-N})    N^{1/2} .$$ To show that assumption~\eqref{eEquiProgression} holds we use
Lemma~\ref{lemma:dukethm} to infer that 
$$\sum_{\substack{ a\in \c A \\ d\mid y_1  } } \chi_N(a)
=\frac{8}{\pi } 
L(1,\chi_{-N})  G_N(d ) N^{1/2}
 +O(d ^{12} N^{1/2-1/30})
,$$
where $G_N$ is multiplicative and defined as 
$$
p^e G_N(p^e) =
\mathds 1(p\neq 2)
2^{\mathds 1(p\mid N)}
\mathds 1(p\mid N\Rightarrow p\equiv 1 \md 4 )
 \l(1     - \frac{1}{p} \l(\frac{-N}{p}\r)  \r) \l(1     - \frac{1}{p} \l(\frac{-1}{p}\r)  \r)
.$$
We have $N^{1/2-\epsilon}\ll M(N)\ll N^{1/2+\epsilon}$ for every fixed $\epsilon>0$
by  Siegel's theorem. Now let $\theta,\xi$ be positive constants that will be fixed later. For all $d \leq M(N)^\theta$ and any fixed positive constant $\epsilon$,
we have 
  $$  d ^{12} N^{1/2-1/30} \ll M(N)^{12 \theta+1-1/15+\epsilon}.$$ Hence,~\eqref{eEquiProgression} holds for some $\xi>0$ as long as 
$12 \theta <1/15$. In particular, it holds when $\theta= 10^{-3}$. 
Assumptions~\eqref{eLowPower}-\eqref{eHighPower}   hold 
due to the bound $G_N(p^e)\ll p^{-e}$ that is valid with an absolute implied constant. One can take $\kappa=10$ in~\eqref{elowaverg}
due to the estimate  $G_N(p) \leq 10/p$ that holds for all primes $p$.

Thus, Theorem~\ref{tMain} shows that 
$$ \sum_{ \substack{ \b x \in  (\Z\setminus \{0\} )^3 \\ x_1^2 + x_2^2 +x_3^2 =N  }} f(|x_1|) 
  \ll  M(N)\prod_{\substack{ 1\ll p\leq M(N)    } } (1-G_N(p) ) 
\sum_{\substack{ a\leq M(N)  }} f(a)   G_N(a)  
.$$ By Lemma~\ref{monteverdi frammenti} we infer that 
$$\prod_{\substack{ 1\ll p\leq M(N)    } } (1-G_N(p) ) 
\sum_{\substack{ a\leq M(N)  }} f(a)   G_N(a)   
\asymp \exp\l( \sum_{p\leq M(N) } (f(p)-1) G_N(p) \r) 
,$$ where the implied constants are independent of $N$.
The sum over $p$ equals 
$$
\sum_{ p\mid N }\l(\frac{-1}{p}\r) \frac{f(p)-1}{p} 
+\sum_{\substack{ p\leq N  }}\frac{f(p)-1}{p}  
+O\Big(1+\sum_{\substack{ p>M(N) \\ p\mid N }} \frac{1}{p}+\sum_{\substack{ M(N)< p \leq N }} \frac{1}{p} \Big)
.$$We have $M(N) \gg N^{1/4}$ by Siegel's theorem, thus, the error term is 
$ \ll 1+ N^{-1/4} \omega(N) $ is bounded, something that   suffices for the proof of the upper bound.
To prove the lower bound we apply Theorem~\ref{thm:lower} in the same manner.

\subsection{Transformation}\label{s:trnvfmt} To transform the sums in Theorem~\ref{thm:3sqrsmultiplc} 
a preliminary step is to show that  for most integer solutions of $x_1^2+ x_2^2+x_3^2=N$  
the common divisors of each pair $(x_i,x_j)$ are typically small. In 
Lemma~\ref{lemghfrommools}  we show that these divisors are frequently smaller than any fixed power of $N$, while in 
Lemmas~\ref{lem:brjsd88y}-\ref{lemcoroohshsls} we show that these divisors are smaller than a   power of $\log N$. The latter task 
combines equidistribution in the form of Lemma~\ref{lemma:dukethm}
with a  ``level-lowering'' mechanism that is grounded on   work  of Brady~\cite{MR3708522}. 
\begin{lemma}\label{lemghfrommools} Fix    any   $s>0$ and $\delta \in (0,1/6)$. Then for any  positive square-free  integer 
$N\equiv 3 \md 4 $ we have   $$ \sum_{c>N^\delta} \sum_{\substack{ \b x \in (\Z\setminus \{0\} )^3, c \mid (x_1,x_2) \\x_1^2 + x_2^2 +x_3^2 =N }} 
(\tau(x_1) \tau(x_2) \tau(x_3) )^s
 \ll  N^{1/2-\delta/2}  L(1,\chi_{-N}),$$ where the implied constant depends at most on $\delta$ and $s$.
\end{lemma}\begin{proof} Since $c^2 \mid x_1^2+x_2^2=N-x_3^2$, we obtain the   upper bound   
$$ \ll_{\epsilon,s}    N^\epsilon  \sum_{c> y }  \sum_{\substack{  |x_3|\leq N^{1/2} \\ c^2\mid N-x_3^2 }  } r_2(N-x_3^2 )
\ll_\epsilon N^{2\epsilon} \sum_{ c> y } \#\{ |x_3|\leq N^{1/2}: c^2\mid N-x_3^2\}.$$ 
We shall now split in two ranges: $$N^\delta<c \leq N^{1/2-\delta}
\hspace{0.5cm} \textrm{ and } \hspace{0.5cm} c > N^{1/2-\delta}
.$$ For the second range we write  $D=(N-x_3^2)/c^2$ and  note that 
$D \leq N^{2\delta}$. Swapping summation thus leads to 
$$  \sum_{ c> N^{1/2-\delta} } \#\{ |x_3|\leq N^{1/2}: c^2\mid N-x_3^2\}\leq 
\sum_{1\leq D \leq N^{2\delta} } \#\{c,x_3 \in \Z: N=x_3^2+ D c^2\}.$$ Since $D>0$, the unit  group of  $\Q(\sqrt{-D})$ is
bounded independently of $D$. From the theory of binary quadratic forms 
we can then infer that  $$ \#\{c,x_3 \in \Z: N=x_3^2+ D c^2\} \ll \sum_{m\mid N} \left( \frac{D}{m}\right) 
\ll _\epsilon N^\epsilon,$$ where the implied constant   
depends only on   $\epsilon$. This gives the overall bound $$\ll N^{3\epsilon+2\delta}\ll
N^{4\epsilon+2\delta} L(1,\chi_{-N})$$ by Siegel's estimate. Using $\delta<1/6$ we see that $2\delta<1/2-\delta$, thus, taking $\epsilon=\delta/8$
gives the bound $N^{1/2-\delta/2} L(1,\chi_{-N})$, which is satisfactory.

We next deal with the first range. Splitting in progressions we get 
$$ \#\{ |x_3|\leq N^{1/2}: c^2\mid N-x_3^2\} \leq \sum_{\substack{ t\in \Z/c^2 \Z  \\ c^2 \mid N-t^2  }} \l(\frac{N^{1/2}}{c^2}+1\r)\ll 
N^\epsilon \l(\frac{N^{1/2}}{c^2}+1\r) .$$ Summing over the range  $ N^\delta<c \leq N^{1/2-\delta} $ this gives 
$\ll N^{1/2-\delta +\epsilon }$, which is acceptable upon choosing a suitably small value for $\epsilon$.   \end{proof}
The proof of the next two results uses crucially that the range $c>N^\delta$ has already been dealt with.
\begin{lemma}\label{lem:brjsd88y} Fix arbitrary $ s>0$ and let     $\beta=60(100 s-1)/7$.
For any $\b c \in \N^3$ and   positive square-free integer $N\equiv 3 \md 8 $  we have
\begin{align*} \sum_{\substack{ \b x \in (\Z\setminus\{0\})^3, c_i\mid x_i \forall i \\ x_1^2+x_2^2+x_3^2=N  }} (\tau(x_1)\tau(x_2)\tau(x_3))^s
&\ll L(1,\chi_{-N}) N^{1/2} (\log \log N)^2 (\log N)^{3\cdot 2^{\beta+1}} \prod_{i=1}^3  \frac{ \tau(c_i)^{\beta+3}}{c_i }  \\
&+  N^{1/2-1/100 } ( c_1 c_2 c_3)^{12} ,\end{align*} 
where  the implied constant depends at most on $\beta $ and $s$.
\end{lemma}
\begin{proof}
The function 
 $H(\delta)=\delta \log_2(\delta^{-1} )+(1-\delta) \log_2(1-\delta)^{-1} $
satisfies  $H(7/6000) > 1/100$. Taking $\delta=7/6000$ 
we   see that the assumption $7\beta +60 = 6000 s$ 
 allows us to use~\cite[Theorem~4]{MR3708522}. This yields the following bound  for the sum over $\b x $ in the lemma:
$$
\ll_{\beta,s} \sum_{\substack{ \b d \in \N^3 \\ d_i \leq N^{\delta/2}\forall i } } (\tau(d_1)\tau(d_2)\tau(d_3) )^\beta
\#\{  \b x \in (\Z\setminus\{0\})^3: x_1^2+x_2^2+x_3^2=N,  [c_i,d_i] \mid x_i\forall i  \},$$where $[\cdot ,\cdot ]$ denotes 
the least common multiple. By Lemma~\ref{lemma:dukethm}   this can be bounded by 
$$ \ll\hspace{-0.3cm} \sum_{\substack{ \b d \in \N^3 \\ d_i \leq N^{\delta/2}\forall i } } \hspace{-0.3cm}
\Big(\prod_{i=1}^3\tau(d_i)^\beta \Big) (\log \log N)^2
\Big( L(1,\chi_{-N}) N^{1/2}\prod_{i=1}^3 \frac{ \tau([c_i,d_i] )}{[c_i,d_i] }+ N^{1/2-1/30}
\prod_{i=1}^3[c_i,d_i]^{12} \Big),$$ where we used the following standard bound for $t=c_1c_2c_3$,
\begin{equation}\label{albinoni oboe concertos kick butt, innit bruv. WHAT DID YA SAY?}
\prod_{p\mid t }\l(1+\frac{1}{p}\r)\ll \prod_{p\mid t }\l(1-\frac{1}{p}\r)^{-1}=\frac{t}{\phi(t)} \ll \log \log t
.\end{equation}
 Using $[c_i,d_i]\leq c_i d_i $ we can see that the second part of this sum is 
$$ \ll   N^{1/2-1/30} \prod_{i=1}^3 c_i^{12} 
\l(
\sum_{ d \leq N^{\delta/2}  } \tau(d)^\beta d^{12} 
 \r)^3
\ll 
N^{1/2-1/30+19.5 \delta } (\log N)^{3(2^\beta- 1)}\prod_{i=1}^3 c_i^{12} 
,$$ which is $ \ll N^{1/2-1/30+20 \delta }  \prod c_i^{12} $. Our choice $\delta=7/6000$ makes sure that this is $ \ll N^{1/2-1/100 }  \prod c_i^{12} $.
The first part of the sum is $ \ll L(1,\chi_{-N}) (\log \log N)^2N^{1/2} \prod \c S(c_i)$, where    $$\c S(c):=
\sum_{  d\leq N  } \tau(d)^\beta
 \frac{ \tau([c,d] )}{[c,d] }
\leq 
\frac{\tau(c) }{c}
\sum_{  d\leq N } \tau(d)^{\beta +1} 
 \frac{ \gcd(c,d) }{d }
.$$
Writing $m=\gcd(c,d) $ and $d= m t $, the sum over $d$ can be seen to be at most  
$$ 
\sum_{m\mid c}  m 
\sum_{  \substack{ d\leq N  \\ m\mid d  }}
 \frac{  \tau(d)^{\beta +1}  }{d }
\leq 
\sum_{m\mid c}     \tau(m  )^{\beta +1}  
\sum_{  \substack{ t\leq N   }}
 \frac{  \tau( t )^{\beta +1}  }{t  }
\ll   \tau(c) ^{\beta+2}  
(\log N)^{2^{\beta+1}}
,$$ which is sufficient.
\end{proof}

\begin{lemma}\label{lemcoroohshsls} Fix any positive $A$ and $s$. Then 
for  any   positive square-free   $N\equiv 3 \md 8 $ we have 
$$ \sum_{c>(\log N)^A} \sum_{\substack{\b x \in  (\Z\setminus\{0\})^3, c\mid (x_1,x_2)  \\x_1^2 + x_2^2 +x_3^2 =N } }
(\tau(x_1)\tau(x_2) \tau(x_3) )^s
\ll_{A,s} 
L(1,\chi_{-N}) N^{1/2}
(\log N)^{\rho(s) -A/2}
 ,  $$ where $\rho(s)=6 \cdot 2^{(6000s-60 )/7 }$
 and the implied constant depends at most on $A$ and $s$.
\end{lemma}\begin{proof} Fix any $\delta>0$. By Lemma~\ref{lemghfrommools} we can discard the contribution of $c>N^\delta$.
For the remaining $c$ we employ 
Lemma~\ref{lem:brjsd88y} 
with $\beta$ defined by 
$7\beta +60 = 6000 s$.
 We    obtain 
\begin{align*}
& \ll
\sum_{(\log N)^A<c \leq N^\delta }
\l(
L(1,\chi_{-N}) N^{1/2}(\log \log N)^2
(\log N)^{3\cdot 2^{\beta+1}}
 \frac{ \tau(c)^{2\beta+6}}{c^2} 
+N^{1/2-1/100 }   c^{24} \r)
\\
& \ll
L(1,\chi_{-N}) N^{1/2}
(\log N)^{3\cdot 2^{\beta+1}-A/2}
 +N^{1/2-1/100 +25 \delta } 
.\end{align*}
Choosing sufficiently small $\delta$ and using 
  Siegel's bound we obtain   
$$N^{ -1/100 +25 \delta } 
\ll N^{ -1/1000}\ll
L(1,\chi_{-N})  
(\log N)^{3\cdot 2^{\beta+1}-A/2}
,$$ which is sufficient.
\end{proof}
Define for $N, m_1,m_2,m_3\in\N$ the function $$
\c R_\b m(N):=
\sum_{\substack{ \b y\in (\Z\setminus\{0\})^3: \sum_i  (m_{j} m_{k} y_i )^2=N \\  \gcd(y_i,  y_j)=1 \forall i\neq j
   } }  
f(  y_1^{\alpha_1} y_2^{\alpha_2} y_3^{\alpha_3})
.$$

\begin{lemma}\label{lem:trsnfmrcohenlenstr}Fix any $A>0$. In the setting of Theorem~\ref{thm:3sqrsmultiplc} we have 
$$
 \sum_{ \substack{ \b x \in  (\Z\setminus \{0\} )^3 \\ x_1^2 + x_2^2 +x_3^2 =N  }} \prod_{i=1}^3 f(x_i)^{\alpha_i} 
\ll \sum_{\substack{\b m\in \N^3, \eqref{eq:condmmm}  \\ \max m_i \leq (\log N)^A  } }  
 \c R_\b m(N)
\prod_{i=1}^3  \tau(m_i) ^{2s} 
+  \frac{  L(1,\chi_{-N}) N^{1/2}}{(\log N)^{A/2-\rho(s)}} ,$$where $\rho(s)$ is as in Lemma~\ref{lemcoroohshsls},
 the implied constant depends at most on  $s,A$ and
\begin{equation}\label{eq:condmmm}\gcd(m_{i},2m_{j})=1 \forall i\neq j\ \ \ 
p\mid m_1m_2m_3 \Rightarrow \l(\frac{N}{p}\r)=1
 .\end{equation}
\end{lemma}\begin{proof}By our assumption $f \leq  \tau^s$  and Lemma~\ref{lemcoroohshsls}   we may write the sum over $\b x $ as  
$$ \sum_{\substack{ \b x \in  (\mathbb Z\setminus\{0\})^3,  x_1^2 + x_2^2 +x_3^2 =N \\ 
\gcd(x_i,x_j) \leq (\log N)^{A} \forall i\neq j  
 }} \prod_{i=1}^3f(x_i)^{\alpha_i}  +O(   L(1,\chi_{-N}) N^{1/2} (\log N)^{-A/2+\rho(s)} ) .$$ For $\{i,j,k\}=\{1,2,3\}$ we let $m_i=\gcd(x_j,x_k)$ 
so that  the $m_{i}$ are coprime in pairs due to  $\gcd(x_1,x_2,x_3)=1$ that can be inferred from 
the fact that $N$ is square-free and $N=\sum_i x_i^2$. Hence, letting $y_i: =x_i/(m_{j} m_{k} )$ we see that 
$m_{i}=\gcd(x_j,x_k)$ is equivalent to  $ 1=\gcd( m_k y_j,  m_j y_k )$.
We obtain 
$$\sum_{\substack{\b m\in \N^3 \\ \gcd(m_{i},m_{j})=1 \forall i\neq j  \\ m_i \leq (\log N)^{A} \forall i } }  
\sum_{\substack{ \b y: \sum_i  (m_{j} m_{k} y_i )^2=N \\  \gcd(y_i,  y_j)=1 \forall i\neq j
 } }  \prod_{i=1}^3f(m_{j} m_{k} y_i )^{\alpha_i}  
+O(   L(1,\chi_{-N}) N^{1/2} (\log N)^{-A/2+\rho(s)}).$$  We omitted the condition $\gcd(y_i,m_i)=1  $ as it is implied by 
the fact that $N$ is square-free and  a sum of integer multiples of 
$m_i^2$ and $y_i^2$. Our assumption $f(ab ) \leq  \tau(a)^s f(b)$   allows us to write   
$$\prod_{i=1}^3  f(m_{j} m_{k} y_i )^{\alpha_i}    \leq   \prod_{i=1}^3 \tau(m_j)^s\tau(m_k)^s f(  y_i )^{\alpha_i}
= f(  y_1^{\alpha_1} y_2^{\alpha_2} y_3^{\alpha_3})  \prod_{i = 1}^3 \tau(m_i)^{2s},
 $$ since the $y_i$ are pairwise coprime and $f$ is multiplicative. The condition that each prime divisor $p$ of $m_i$ must satisfy 
$(\frac{N}{p})=1$ comes from the fact that each $m_i$ divides two of the coefficients of $ \sum_i  (m_{j} m_{k} y_i )^2$ and is coprime to the third.
Finally, if one of the  $m_i$  is even, then $4$ divides $N-(m_jm_k y_i)^2$, which is impossible owing to $N\equiv 3 \md 4 $.
\end{proof}
\subsection{Level of distribution}\label{ss:lebvelfdisnv78} Throughout this subsection 
  $\b m$ is  a fixed vector in $\N^3$ satisfying~\eqref{eq:condmmm}.
For   positive integers $d, N$     define 
$$C_d(N):= \#\left\{\b y \in \Z^3:  \begin{array}{l}
(m_{2} m_{3} y_1 )^2+(m_{1} m_{3} y_2 )^2+(m_{1} m_{2} y_3)^2=N,\\
 \gcd(y_i,y_j)=1 \forall i\neq j , \ \ \ d \mid y_1y_2y_3 \end{array} \right\}
.$$ The main result  is Lemma~\ref{lemma:putting the pudding}; it 
gives  a level of distribution result for $C_d(N)$ that  will subsequently be fed into 
 Theorem~\ref{tMain} to bound  $\c R_\b m(N)$. 

We start with a sieving argument that deals with the coprimality of the $y_i$.
\begin{lemma}\label{lem:costanza bernini}Keep the setting of Theorem~\ref{thm:3sqrsmultiplc} and fix any $\delta\in (0,1/9)$.
For all $\b m $ as  in~\eqref{eq:condmmm} and all $d\in \mathbb N$  we have  $$C_d(N)=
\sum_{\substack{ \b d \in \N^3 , d=d_1d_2d_3 \\\gcd(d_i,d_j)=1 \forall i\neq  j\\\gcd(d_i,m_i ) =1 \forall i  }}
\sum_{\substack{ \b b \in \N^3, \max b_i \leq N^{\delta_i} \forall i
 \\\gcd(b_i,b_j)=1\forall i\neq  j   \\  \gcd(b_i,d_i   m_j)=1 \forall i\neq  j    } } 
 \mu(b_1) \mu(b_2) \mu(b_3)  C_{\b b  , d}(N) +O( N^{1/2-\delta/400 } L(1,\chi_{-N}) )
,$$ where  $ \delta_1=\delta/100,\delta_2=\delta/10,\delta_3=\delta$,   the quantity
 $ C_{\b b  , d}(N) $ is given by $$ \#\{\b t \in \Z^3:  
N= (m_{2} m_{3}  [d_1,b_2b_3] t_1 )^2+(m_{1} m_{3} [d_2,b_1 b_3] t_2 )^2+(m_{1} m_{2} [d_3,b_1 b_2] t_3 )^2  \}$$
and the implied constant depends at most on $\delta$.
\end{lemma}\begin{proof}Since $y_i$ are coprime in pairs in $C_d(N)$, we can write $d=d_1d_2d_3$ where $d_i\mid y_i$ and the 
$d_i$ are coprime in pairs. Then, $C_d(N)$ becomes $$\sum_{\substack{ \b d \in \N^3 , d=d_1d_2d_3 \\\gcd(d_i,m_id_j)=1 \forall i\neq  j }}
 \#\left\{\b y \in (\Z\setminus\{0\})^3:  \begin{array}{l}
(m_{2} m_{3} y_1 )^2+(m_{1} m_{3} y_2 )^2+(m_{1} m_{2} y_3)^2=N,\\
 \gcd(y_i,y_j)=1 \forall i\neq j , \ \ \ d_i \mid y_i \forall i \end{array} \right\}
  .$$ The condition $\gcd(d_i,m_i)=1$ comes   from the fact that 
  $N$ is square-free and  the   sum of squares equation.
 We now use the expression $\sum_{b_1 \mid (y_2,y_3)} \mu(b_1)$ to detect the coprimality of $y_2$ and $y_3$.
The contribution of $b_1>N^{\delta/100}$ will then be at most 
$$\tau_3(d) \sum_{b_1>N^{\delta/100} }  \# \{\b t \in \Z^3:  
t_1 ^2+t_2^2+t_3^2=N, b_1\mid (t_2,t_3)    \} \ll \tau_3(d) N^{1/2-\delta/200} L(1,\chi_{-N}) $$ 
by Lemma~\ref{lemcoroohshsls}. We have $d \leq N^3$ due to  $d\mid y_1y_2 y_3$, hence, 
 the bound $\tau_3(d) \ll   N^{\delta/400}$ shows that the contribution is $  \ll  N^{1/2-\delta/400} L(1,\chi_{-N}) $.
Next, we use  $\sum_{b_2 \mid (y_1,y_3)} \mu(b_2)$ to detect the coprimality of $y_1$ and $y_3$. The contribution of $b_2>N^{\delta/10}$
is $$\ll \tau_3(d) \sum_{\substack{ b_1\leq N^{\delta/100} \\ b_2>N^{\delta/10} } }  \# \{\b t \in \Z^3:  
t_1 ^2+t_2^2+t_3^2=N,  b_2\mid (y_1,y_3)    \} \ll \tau_3(d) N^{1/2+\delta/100 -\delta/20} L(1,\chi_{-N}) $$
 by Lemma~\ref{lemcoroohshsls}. This can be seen to be 
$ \ll  N^{1/2 -\delta/150} L(1,\chi_{-N}) $ as before. Finally, using  $\sum_{b_3 \mid (y_1,y_2)} \mu(b_3)$, we can see 
that the range $b_3> N^{\delta}$ contributes 
 $$\ll \tau_3(d) \sum_{\substack{ b_1\leq N^{\delta/100},  b_2\leq N^{\delta/10}
\\ b_3>N^{\delta}  } }  \# \{\b t \in \Z^3:   t_1 ^2+t_2^2+t_3^2=N, 
b_3\mid (t_1,t_2)    \} \ll   N^{1/2-\delta/10 } L(1,\chi_{-N}) .$$
 We thus obtain the expression claimed in the lemma. The conditions of the form   $\gcd(b_1, b_2 b_3 d_1 m_2m_3)=1$ in the lemma 
come from the fact that $N$ is  square-free. Finally, the vectors $\b t$ having $t_i=0$ for some $i$ 
contribute at most  $$ \ll \tau_3(d) N^{\delta_1+\delta_2+\delta_3} r_2(N)\ll N^{2\delta},$$ which is acceptable by the assumption $\delta<1/9$.
\end{proof}
We next apply Lemma~\ref{lemma:dukethm}.  Denote $$b=b_1b_2b_3, \ \ 
m=m_1m_2m_3
\ \ \textrm{ and } \ \ 
 \mathfrak C_{d}:= \prod_{\substack{ p\equiv 3 \md 4 \\ p\mid (d,N)  } } p .$$
\begin{lemma}\label{lem:Cesar Frank piano quintet in F minor}
Keep the setting of Lemma~\ref{lem:costanza bernini} and fix any $\varpi>0$.
For all $d\in \mathbb N$ and $\b m\in \N^3$ as in~\eqref{eq:condmmm} with the additional restriction $\max m_i \leq (\log N)^\varpi$
we have  $$C_d(N)=
 \frac{8}{\pi }  L(1,\chi_{-N})  N^{1/2}
M_1 M_2+ O( d^{12} N^{1/2+\max\{50\delta-1/30,-\delta/800  \}}(\log N)^{100\varpi }  L(1,\chi_{-N}) )
,$$ where the implied constant depends at most on $\delta$ and $\varpi$. Here 
$$ M_1=  \frac{ \mathds 1(2\nmid d )  }{d} 
\frac{  2^{\omega(m)}}{m^2}  
  \prod_{p\mid  m} \l( 1-\frac{1}{p} \l(\frac{-1}{p}\r) \r)
$$
and
\begin{align*} 
M_2&=\sum_{ \b b \in \N^3 } 
2^{\#\{p\mid b :p\nmid m \} } \frac{ \mu(b)  }{b^2} 
2^{\#\{p\mid d: p\nmid b m \}}  
\prod_{p\mid b, p\nmid m} \l( 1-\frac{1}{p} \l(\frac{-1}{p}\r) \r)
\frac{2^{\#\{p\mid (d,b )\}}}{2^{\#\{p\mid d: p\nmid bmN\}}}
\\ & \times 
 \gcd(d,b )
3 ^{\#\{p\mid d: p\nmid bm \}}
 2 ^{\#\{p\mid d , p\mid m , p\nmid b\}} 
 \prod_{ \substack{ p\mid d \\ p\nmid b m }  } 
\l(1     - \frac{1}{p} \l(\frac{-1}{p}\r)  \r)
\l( 1-\frac{1}{p} \l(\frac{-N}{p}\r) \r)
,\end{align*} 
where the sum is over  $\b b$ satisfying the further conditions   
$$ \mathfrak C_{d} \mid b_1 b_2 b_3 m , 
\gcd(b_i,2b_jm_j)=1\forall i\neq  j ,
$$ and $(\frac{N}{p})=1$ for all primes $ p\mid b_1 b_2 b_3  $ with $ p\nmid m$.
\end{lemma}
\begin{proof}We employ Lemma~\ref{lemma:dukethm} with $c_i=m_j m_k [d_i,b_j b_k]$
to estimate  $C_d(N)$ in   Lemma~\ref{lem:costanza bernini}. The error term is $$ \ll
\tau_3(d)d^{12} N^{1/2-1/30}(\log N)^{100\varpi } \prod_{i=1}^3  \sum_{  b \leq N^{\delta_i}} b^{24} 
 \ll
d^{12} N^{1/2-1/30+50\delta}(\log N)^{100\varpi } 
L(1,\chi_{-N}) $$ by Siegel's bound   and  $\tau_3(d)\ll N^{\delta-\delta_1-\delta_2}L(1,\chi_{-N})$ that is implied by $d\leq N^3$.

To deal with the main term let us recall that the  $m_i$ are pairwise coprime and use the 
 coprimality conditions on the $b_i,d_i$ to see that~\eqref{qweeebok1} is always met. 
Denote $b:=b_1b_2b_3$. Note that a prime  $p$ divides exactly two of the $c_i$ 
if and only if $p\mid bm$. In addition, $p$ divides exactly one of $c_i$
if and only if $p$ divides $d$ but not $bm$.
We get the main term  
$$  \frac{8}{\pi }  L(1,\chi_{-N})
N^{1/2}
\frac{ \mathds 1 (2\nmid m )}{m^2}\mathfrak K
 \prod_{p\mid m}\l( 1-\frac{1}{p} \l(\frac{-N}{p}\r) \r)
,$$
where $\mathfrak K$ is the sum 
\begin{align*}
& \Osum 
2^{\omega(m b  )}
\frac{ \mu(b_1) \mu(b_2) \mu(b_3)  }{(b_1b_2b_3)^2}
\prod_{\substack{p\mid b_1b_2b_3 \\p\nmid m }}\l( 1-\frac{1}{p} \l(\frac{-N}{p}\r) \r)
\frac{ \mathds 1(2\nmid d ) 2^{\#\{p\mid d: p\nmid b m  \}}  }{d}   \mathfrak F(d)
\\ \times & \mathds 1(p\mid (d, N), p\nmid b  m  \Rightarrow p\equiv 1 \md 4 )
 \prod_{ \substack{ p\mid d \\ p\nmid b  m }  }
\l(1     - \frac{1}{p} \l(\frac{-1}{p}\r)  \r)
\l( 1-\frac{1}{p} \l(\frac{-N}{p}\r) \r)
\l( \frac{1 }{2}\r)^{\mathds 1 (p\nmid N ) }
\end{align*} with $\Osum$ taken over $\b b \in \N^3$ satisfying 
$b_i \leq N^{\delta_i}$ for all $i$,
$\gcd(b_i,2b_jm_j)=1$ for all  $i\neq  j$, and, 
with the further property that  each prime divisor $p$ of $b$ that does not divide $m$ must satisfy  
$(\frac{N}{p})=1$. The multiplicative function $ \mathfrak F(d)$ is defined as  $$   \sum_{\substack{ \b d \in \N^3 , d=d_1d_2d_3 }}
\gcd(d_1,b_2b_3)\gcd(d_2,b_1 b_3)\gcd(d_3,b_1 b_2),$$ where the sum is subject to 
$\gcd(d_i,d_jm_i b_i)=1$ for all $i\neq j$. To analyse it at prime powers $p^\alpha$   we use 
 that   $d_i$ are coprime to infer  that  $ \mathfrak F(p^\alpha)$ equals $$\gcd(p^\alpha,b_2b_3) \mathds 1 (p\nmid b_1 m_1 )
+\gcd(p^\alpha,b_1 b_3) \mathds 1 (p\nmid b_2 m_2 )
+\gcd(p^\alpha,b_1 b_2) \mathds 1 (p\nmid b_3 m_3 )
.$$ Since   $b_i$ are coprime in pairs and square-free we see that if $p\mid b_1b_2 b_3$ then the above becomes 
$2p$ because $\gcd(b_i, m_j  )=1$ for all $i\neq j $.   If $p\nmid b_1b_2 b_3 m_1 m_2 m_3 $ then the sum becomes $3$.
If  $p\nmid b_1b_2 b_3   $ and $p\mid   m_1 m_2 m_3 $  then it becomes $2$.
Thus, $\mathfrak F(d)$ equals   $$  2^{\#\{p\mid (b,d)\}} \gcd(b,d)
3^{\#\{p\mid d:p\nmid bm \}} 2^{\#\{p\mid (d,m):p\nmid b \}}
.$$ Using~\eqref{albinoni oboe concertos kick butt, innit bruv. WHAT DID YA SAY?} we see that  
the contribution of $\b b $ with $b_i> N^\delta_i$ for some $i$ is  $$
\ll  L(1,\chi_{-N}) N^{1/2}(\log \log N)^3   \tau(d) 6^{\omega(d)}  d \sum_{\substack{ \b b \in \N^3 \\\exists i: b_i >N^{\delta_i} }} \frac{2^{\omega(b_1b_2b_3)}}{(b_1b_2b_3)^2}
\ll   L(1,\chi_{-N}) N^{1/2}    d^2 N^{-\frac{1}{2\min \delta_i}} ,$$ which is acceptable since $\min \delta_i > \delta/800$.
To conclude the proof we note that the condition $ p\mid d,p\mid N, p\nmid b m \Rightarrow p\equiv 1 \md 4  $   is equivalent to 
$ \mathfrak C_{d} \mid bm$. \end{proof} 
Finally, we simplify the main term in  Lemma~\ref{lem:Cesar Frank piano quintet in F minor}.
The error term will be  obtained by taking $\delta=80/120003$. Denote for a prime $p$, $$ c_p=1-\frac{1}{p}\l(\frac{-1}{p}\r).$$
\begin{lemma}\label{lemma:putting the pudding} Fix any $\varpi>0$.
For all $ \b m \in \N^3$ as in~\eqref{eq:condmmm}   with $\max m_i \leq (\log N)^{\varpi} $, all
square-free positive integers  $N\equiv 3 \md 8 $ and all   $d\in \N$  we have  $$
C_d(N)=M(N)g_N(d) +O(  d^{12} N^{1/2 -1/1200030
 }(\log N)^{100\varpi }  L(1,\chi_{-N}) )
,$$ where the implied constant depends at most on $\varpi$. Further,   $$M(N)
= \frac{8}{\pi}  L(1,\chi_{-N})   N^{1/2}  \frac{ 2^{\omega(m_1m_2m_3)}}{(m_1m_2m_3)^2}  
\prod_{p\mid m_1m_2m_3 } c_p \l(1-\frac{1}{p^2}\r)
\prod_{ \substack{   p\nmid m_1m_2m_3 \\ (\frac{N}{p})=1 } } \l(1-\frac{6c_p}{p^2}  \r) 
$$ and  \begin{align*} g_N(d)= &
\mathds 1 (p\mid (d, N) \Rightarrow p\equiv 1 \md 4  )
\frac{ \mathds 1(2\nmid d )  }{d} 
2^{\#\{p\mid d: p\mid  m_1m_2m_3N\}} 3 ^{\#\{p\mid d: p\nmid m_1m_2m_3 \}}
\\ \times  & \prod_{\substack{  p\mid m_1m_2m_3 \\ p\mid d } }  \l(1+\frac{1}{p}\r)^{-1}
\prod_{ \substack{ p\nmid  m_1m_2m_3 \\p\mid d,  (\frac{N}{p})=1 } }
\frac{1}{\l(1-\frac{6c_p }{p^2}\r)\l(1-\frac{4}{pc_p}\r)}
 \prod_{ \substack{ p\mid d\\ p\nmid m_1m_2m_3     }  }  c_p\l(1-\frac{1}{p} \l(\frac{-N}{p} \r)\r)
.\end{align*} \end{lemma} \begin{proof}
Let $\mathfrak G(b)$ be the number of $\b b \in \N^3$   with 
$b= b_1b_2 b_3 $ and $\gcd(b_i,b_jm_j)=1$  for all $i\neq  j   $.
Using $2^{\#\{p\mid d: p\nmid b m \}}  2 ^{\#\{p\mid d , p\mid m , p\nmid b\}} 2^{\#\{p\mid (d,b )\}}=2^{\omega(d)}$ we can write $$M_2=
M_3 3 ^{\#\{p\mid d: p\nmid m \}}
\frac{2^{\omega(d) } }{2^{\#\{p\mid d: p\nmid  mN\}}}
\prod_{ \substack{ p\mid d, p\nmid m     }  }  c_p \l( 1-\frac{1}{p} \l(\frac{-N}{p}\r) \r)
,$$where $M_3$ is given by 
$$\sum_{\substack{   b \in \N,  2\nmid  b, \mathfrak C_{d} \mid b m
\\ p\mid b , p\nmid m  \Rightarrow (\frac{N}{p})=1   } } 
2^{\#\{p\mid b :p\nmid m \} } \frac{ \mu(b)   \gcd(d,b ) }{b^2} 
 \frac{ 2^{\#\{p\mid b: p\nmid  mN, p\mid d \}}  }{ 3 ^{\#\{p\mid b: p\nmid m, p\mid d  \}}  }
\mathfrak G(b)  
 \prod_{ \substack{p\mid b, p\nmid m  \\ p\mid d    }  } c_p^{-2}
 \prod_{p\mid b, p\nmid m} c_p
.$$

For a prime $p$ we have $\mathfrak G(p)=\mathds 1 (p\nmid m_1m_2)
+\mathds 1 (p\nmid m_1m_3)+\mathds 1 (p\nmid m_2 m_3)$. Since the $m_i$ are coprime in pairs, $\mathfrak G(p)$ becomes 
 $1$  or $3$ according to whether 
 $p$ divides $m$ or not. 
Hence,  $\mathfrak G(b)=3^{\#\{p\mid b: p\nmid m\}}$ for all square-free $b$, thus, 
$M_3$ can be written as $$
 \sum_{\substack{   b \in \N 
, 2\nmid  b,  \mathfrak C_{d} \mid bm
\\  p\mid b , p\nmid m  \Rightarrow (\frac{N}{p})=1  
  } }  6^{\#\{p\mid b :p\nmid m \} } \frac{ \mu(b)   \gcd(d,b ) }{b^2} 
 \frac{ 2^{\#\{p\mid b: p\nmid  mN, p\mid d \}}  }{ 3 ^{\#\{p\mid b: p\nmid m, p\mid d  \}}  }
 \prod_{ \substack{p\mid b, p\nmid m  \\ p\mid d    }  }  c_p^{-2}
  \prod_{p\mid b, p\nmid m} c_p
.$$Let us show  that if the sum over $b$ is non-empty then  $ \mathfrak C_{d}=1$. To see that, assume there is a prime  
$p\mid  \mathfrak C_{d}$. Then the condition $p\mid \mathfrak C_{d} \mid bm $ implies that  $p\mid m $ or $p\nmid m$ and $p\mid b $.
In the first case, the condition present in $M_1$ shows that $(\frac{N}{p})=1$, which violates the condition $p\mid \mathfrak C_{d} \mid N$.
In the second case, we have $p\nmid m$ and $p\mid b $, hence, the condition in the sum over $b$ shows that $(\frac{N}{p})=1$, which is a contradiction.

Now, factor the square-free $b$ as $b_0b_1$, where $b_0\mid m$ and $b_1$ is coprime to $m$. We can thus write $M_3=M_4 M_5$, where  
$$ M_4=   \sum_{   b_0  \mid m } 
\frac{ \mu(b_0  )   \gcd(d,b_0   ) }{b_0^2 }  =
\prod_{\substack{  p\mid m \\ p\mid d } } \l(1-\frac{1}{p}\r)
\prod_{\substack{   p\mid m \\ p\nmid d}} \l(1-\frac{1}{p^2}\r)
$$   and $M_5$ is given by 
\begin{align*}  \sum_{\substack{    b_1 \in \N,
\gcd(b_1,2m)=1
\\ 
p\mid   b_1   \Rightarrow (\frac{N}{p})=1  
} } 
&  6^{\#\{p\mid b_1 \} } \frac{ \mu( b_1)   \gcd(d,  b_1 ) }{ b_1^2} 
\l(\frac{2}{3}\r)^{\#\{p\mid b_1:  p\mid d \}}
 \prod_{p\mid b_1} c_p
 \prod_{ \substack{p\mid b_1  \\ p\mid d   }  }c_p^{-2}
 ,\end{align*} where we used the conditions   $b_0 \mid m$ and $\gcd(b_1, m)=1$ to infer that $b_0,b_1$ are coprime and thus split $\mu(b_0 b_1)$.
The Euler product for $M_5$ equals $$
\prod_{ \substack{  p\nmid 2m \\ (\frac{N}{p})=1 } } 
\l(1-\frac{6c_p}{p^2} \r)
\prod_{ \substack{ p\mid d, p\nmid 2m \\ (\frac{N}{p})=1 } }  
\frac{1}{\l(1-\frac{6c_p}{p^2} \r)\l(1-\frac{4}{p c_p} \r)},$$  which concludes the proof.
\end{proof}

\subsection{The proof of Theorem~\ref{thm:3sqrsmultiplc}}\label{ss:finalproofthrmes123} 
We  apply Theorem~\ref{tMain} with $\mathcal A$ being the set of vectors $\b y \in (\Z\setminus \{0\} )^3 $ satisfying 
 $\gcd(y_i,  y_j)=1$ for all  $i\neq j $ and  $c_a=|y_1y_2y_3|$. Further, we let  $T=N$ and 
 $  \chi_N(a)=\mathds 1 _{\{N\}} ((m_{2} m_{3} y_1 )^2+(m_{1} m_{3} y_2 )^2+(m_{1} m_{2} y_3)^2) $.
To verify  assumption~\eqref{eEquiProgression}  we   use Lemma~\ref{lemma:putting the pudding}. Note that 
 $2^{\omega(m)} \prod_{p\mid m}(1-1/p)\geq 1$, hence  \begin{equation}\label{eq:finallyending}\frac{N^{1/2} L(1,\chi_{-N})}{(\log N)^{6\varpi}} \leq
\frac{N^{1/2} L(1,\chi_{-N})}{(m_1m_2m_3)^2} \ll
 M\ll N^{1/2} L(1,\chi_{-N})\end{equation} with absolute implied constants.
Fix any strictly positive constants $\xi$ and $ \theta$ satisfying $12\theta+\xi <1/600015$.
For any  positive integer $d\leq M^\theta$, the error term in Lemma~\ref{lemma:putting the pudding} is  
 $$ \ll M^{12\theta}  N^{1/2} L(1,\chi_{-N}) N^{ -1/1200030}(\log N)^{100\varpi }
 \ll M^{1+12\theta}   N^{ -1/1200030}(\log N)^{106\varpi }  $$ by the lower bound~\eqref{eq:finallyending}.
Using the upper bound of the same inequality we obtain 
$$  \ll M^{1 -1/600015+12\theta} L(1,\chi_{-N})^{1/2400060} (\log N)^{106\varpi }  
\ll M^{1 -1/600015+12\theta}   (\log N)^{1+106\varpi }   $$
by the bound     $L(1,\chi_{-N}) \ll \log N$. The error term is $O(M^{1-\xi})$ as we have chosen $\xi$ so that $12\theta+\xi <1/600015$.
This verifies  assumption~\eqref{eEquiProgression} of Theorem~\ref{tMain}. The remaining assumptions 
are easily seen to hold   since the function $g_N$ in Lemma~\ref{lemma:putting the pudding}  satisfies  
$p^e g_N(p^e)=O(1)$ for all $e\geq 1$ and primes $p$ with an absolute implied constant.
Hence, one has for all $\b m $ with $\max m_i \leq (\log N)^A$ 
$$   
\sum_{\substack{ \b y\in (\Z\setminus\{0\})^3: \sum_i  (m_{j} m_{k} y_i )^2=N \\  \gcd(y_i,  y_j)=1 \forall i\neq j
   } }   f(  |y_1y_2y_3|)    \ll  M(N) T(M(N) ) ,$$ where 
$$T(y)=\prod_{\substack{ 1\ll p\leq y   } } (1- g_N(p) )  \sum_{\substack{ a\leq y  }} f(a)  g_N(a)  $$  and 
the implied constant depends at most on $s$ and $\varpi$.    
We have $T(y)\asymp \exp(S(y))$ by Lemma~\ref{monteverdi frammenti}, where $ S(y)$ is  $$ 
6\sum_{\substack{ p\mid N  \\ p\equiv 1 \md 4    }} \frac{ f(p) -1}{p} 
 +3\sum_{\substack{ p\leq y \\  p\nmid N      }} \frac{ f(p) -1}{p} 
+O\Big(1+\sum_{ p\mid m_1m_2m_3   } \frac{1}{p}+\sum_{\substack{ p\mid N \\ p>y    }} \frac{1}{p} \Big)
.$$ The main term is   $$  3 \sum_{ p\mid N     } \l(\frac{-1}{p} \r) \frac{ f(p) -1}{p} 
 +3\sum_{ p\leq y     } \frac{ f(p) -1}{p}  +O\Big(  \sum_{\substack{ p\mid N \\ p>y    }} \frac{1}{p} \Big) .$$
The sum over $p\mid N, p>y$ is $\ll \omega(N)/y \ll (\log N)/y $. Thus,
with $c_f(N)$   as in~\eqref{def:obscura valediction},
 there exists a  positive constant $\nu =\nu(s)$ such that  for all $y\geq \log N$
one has  $$\prod_{p\mid m_1m_2m_3} \l(1+\frac{1}{p}\r)^{-\nu} \ll T(y) c_f(y)^{-3}    \exp\l(-3 \sum_{p\leq y} \frac{f(p)-1}{p} \r)
\ll \prod_{p\mid m_1m_2m_3} \l(1+\frac{1}{p}\r)^{\nu} .$$  
 Injecting the upper bound into Lemma~\ref{lem:trsnfmrcohenlenstr} and using that $M(N) \leq N$ we get  
\begin{align*} 
\sum_{ \substack{ \b x \in  (\Z\setminus \{0\} )^3 \\ x_1^2 + x_2^2 +x_3^2 =N  }} \prod_{i=1}^3 f(|x_i|)  \ll 
\frac{  L(1,\chi_{-N}) N^{1/2}}{(\log N)^{A/2-\rho(s)}}
+ L(1,\chi_{-N}) N^{1/2}  c_f(N)^{3}    \exp\l(3 \sum_{p\leq N} \frac{f(p)-1}{p} \r)
 \c S\end{align*} 
where 
$$
\c S=\sum_{\substack{\b m\in \N^3, \eqref{eq:condmmm}  \\ \max m_i \leq (\log N)^A  } }  \prod_{i=1}^3  \frac{\tau(m_i) ^{2s}}{m_i^2} 
\prod_{p\mid m_i} \l(1+\frac{1}{p}\r)^{\nu}.
$$ 
Since $\prod_{p\mid m} (1+1/p) \leq \tau(m)$ we can see that $\c S$ is bounded. 
Enlarging the value of $A$   allows  the logarithmic exponent $A/2-\rho(s)$ to exceed any given number 
and it thus completes the proof of the upper bound in Theorem~\ref{thm:3sqrsmultiplc}  when each $\alpha_i$ is $1$.

To prove the lower bound  in Theorem~\ref{thm:3sqrsmultiplc}  we   note that 
$$ \sum_{ \substack{ \b x \in  (\Z\setminus \{0\} )^3 \\ x_1^2 + x_2^2 +x_3^2 =N  }} \prod_{i=1}^3 f(|x_i|) 
 \geq 
\sum_{ \substack{ \b x \in  (\Z\setminus \{0\} )^3, \ \sum_i x_i^2=N
\\
\gcd(x_i,x_j)=1 \forall i\neq j    }} \prod_{i=1}^3 f(|x_i|) 
$$ and apply    Theorem~\ref{thm:lower} to estimate the right-hand side sum. This has a level-of-distribution assumption  that can be verified using the case 
 $m_1=m_2=m_3=1$ of Lemma~\ref{lemma:putting the pudding}. 
The residual stages in the proof are indistinguishable to those for the upper bound.

\subsection{General Diophantine equations} \label{s:aplntratinpts}
The proof of Theorem~\ref{thm:spitfire} is based on an application of Theorem~\ref{tMain}. 
This has specific assumptions; we start by verifying the ones related to the level of distribution in Lemma~\ref{lem:birlevel} 
and proceed by verifying the ones related to  the growth of the sieve density function in Lemmas~\ref{lem:30thousandfeet}-\ref{lem:deathofalgebra}.

\begin{lemma}\label{lem:bareboneschocolate}Let $F\in \Z[x_0,\ldots, x_n]$ be as in Theorem~\ref{thm:spitfire}.
Then for every $0\leq i \leq n $ we have $\#\{   \b x \in (\Z\cap [-B,B])^{n+1}:  F(\b x)=0, x_i=0 \}\ll B^{n-d}$, where the implied constant 
depends only on $i$ and $F$. \end{lemma}\begin{proof} It is necessary to 
recall the   definition of the Birch rank $\mathfrak B(g)$ of a  polynomial $g\in \Z[x_0,x_1,\ldots, x_m]$, where $m\in \mathbb N$.
Denoting   the homogeneous part of $g$ by $g^\flat$, the number $\mathfrak B(g)$ is defined as  the codimension of the affine variety in $\mathbb C^{m+1}$
 given by $\nabla g^\flat(\b x)=0$. Note that $\mathfrak B(g)=m+1$ when $g^\flat$ is smooth. Returning to the proof of our theorem we note that 
setting $x_i=0$ in the polynomial $F(\b x)$ will  produce  a homogeneous polynomial $F_i$ in at most $n-1$ variables.
We claim that $F_i$ will have degree $d$. If not, then   $F_i$ must vanish identically, which can only happen when each monomial of $F_i$ contains $x_i$;
hence, $x_i$ would divide $F(\b x )$ and this would contradict the assumed smoothness of $F$.

By~\cite[Lemma~3.1]{MR3867326} one has $\mathfrak{B}(F_i)\geq \mathfrak{B}(F)-2$, since 
in the notation of~\cite{MR3867326}  one has $|\b j|_1=1$ and $R=1$. Recalling that $F$ is smooth one sees that $\mathfrak{B}(F)=n+1$.
Hence,  $$ \mathfrak{B}(F_i) \geq  n-1 \geq (d-1) 2^{d-1}=(\deg(F_i)-1) 2^{\deg(F_i)-1} .$$ Hence, $F_i$
satisfies the assumption on the number of variables for Birch's work~\cite{MR150129}. In particular,~\cite[Equation~(4), page 260]{MR150129}
applied to $F_i$ gives 
$$  \#\{   \b x \in (\Z\cap [-B,B])^{n}:  F(x_0,x_1,\ldots, x_{i-1}, 0,x_{i+1}, \ldots, x_n )=0 \}  \ll B^{n-\deg(F_i)}= B^{n-d},
$$ which is sufficient. 
\end{proof}

For a prime $p$ define 
$$
\sigma_p:= \lim_{m\to+\infty }  \frac{ \#\{ \mathbf x \in (\Z/p^m\Z)^{n+1} :F(\mathbf x)\equiv 0 \md {p^m}    \}}{p^{mn}}.
$$   
Combining~\cite[Equation~(20), page 256]{MR150129} with~\cite[Lemma~7.1]{MR150129}
we see  that   the limit converges and,  furthermore, 
\begin{equation}\label{eq:renfrewstreet}
\sigma_p=1+O(p^{-1-\lambda})\end{equation} for some positive $\lambda=\lambda(F)$. The corresponding density over $\mathbb R$ is given by 
$$\sigma_\infty:=\int_{-\infty}^{+\infty} 
 \int_{\b x \in [-1,1]^{n+1} } \exp (2 \pi i \gamma F(\b x ) ) \mathrm d \b x \mathrm d \gamma
.$$ It converges due to~\cite[Lemma~5.2]{MR150129}. Finally, we let 
$$\sigma(F):=\sigma_\infty\prod_p\sigma_p.$$

\begin{lemma}\label{lem:birlevel} Keep the assumptions of 
Theorem~\ref{thm:spitfire} and assume that $F=0$ has a $\Q$-point. 
Then there exist positive constants $\Xi, \beta$, both of which depend on $F$, such that 
for all $q\in \mathbb N$ and $B\geq 1 $ with 
$q\leq B^\Xi$  one has $$\#\{   \b x \in (\Z\setminus\{0\})^{n+1}: \max |x_i|\leq B , F(\b x)=0, q \mid x_0 \cdots x_n  \}
=\sigma(F) B^{n+1-d} (h_F (q)+O(B^{-\beta})), $$ where  the function $h_F:\N \to [0,\infty)$ is  defined by 
$$h_F(q)=\prod_{p\mid q} 
 \frac{ 1}{ \sigma_p }
\lim_{m\to+\infty } 
\frac{ \#\left\{ \mathbf x \in (\Z/p^m\Z)^{n+1} :p^m\mid F(\mathbf x) , p^{v_p(q)} \mid x_0 \cdots x_n   \right \}}{p^{mn}}
$$ and 
 the implied constant is independent of $q$.
\end{lemma}
\begin{proof} Adding back the terms for which  $x_0\cdots x_n=0 $ shows that  
the counting function equals$$\#\{   \b x \in (\Z\cap [-B,B])^{n+1}:  F(\b x)=0,  \mathbf x \equiv \mathbf t \md q \}
+O (\mathcal E ),$$where $\mathcal E =   \#\{   \b x \in (\Z\cap [-B,B])^{n+1}:  F(\b x)=0, x_0\cdots x_n=0 \}$. By 
Lemma~\ref{lem:bareboneschocolate} we have $\mathcal E =O(B^{n-d})$, which is satisfactory.
To deal with the main term we  partition in progressions to convert it  into 
 $$
\sum_{\substack{\mathbf t \in (\Z/q\Z)^{n+1} \\  t_0 \cdots t_n \equiv 0 \md q } }
\#\{   \b x \in (\Z\cap [-B,B])^{n+1}:  F(\b x)=0,   \mathbf x \equiv \mathbf t \md q\}
 .$$ We now employ~\cite[Lemma~4.4]{MR4402657} to deduce that 
the cardinality equals 
$$
 \sigma_\infty B^{n+1-d} 
\prod_{p\mid q} \sigma_p(\mathbf t, p^{v_p(q)})
\prod_{p\nmid q} \sigma_p 
+O(B^{n+1-d-\eta}q^{M})
 ,$$
where $\eta,M$ are positive constants that depend only on $F$, 
$$\sigma_p(\mathbf t, p^k)
:=  \lim_{m\to+\infty } 
\frac{ \#\left\{ \mathbf x \in (\Z/p^m\Z)^{n+1} :F(\mathbf x)\equiv 0 \md {p^m} ,\b x \equiv \b t \md {p^k} \right \}}{p^{mn}}
$$ and     the implied constant depends at most on $F$. The contribution of the error term is 
$$\ll B^{n+1-d-\eta}q^{M}
\sum_{\substack{\mathbf t \in (\Z/q\Z)^{n+1} \\  t_0 \cdots t_n \equiv 0 \md q } }
1\ll B^{n+1-d-\eta}q^{M+n+1}
.$$ Letting $\Xi:=\frac{\eta}{2(M+n+1)}$, the assumption $q\leq B^\Xi$ implies that   $q^{M+n+1} \leq B^{ \eta/2}  $, 
thus the error term is $O(B^{n+1-d-\eta/2})$, which is satisfactory. The main term contribution becomes 
 $$ \sigma(F) B^{n+1-d}  
\sum_{\substack{\mathbf t \in (\Z/q\Z)^{n+1} \\  t_0 \cdots t_n \equiv 0 \md q } }
\prod_{p\mid q}\frac{ \sigma_p(\mathbf t, p^{v_p(q)})}{ \sigma_p }
,$$ where we have used the fact that $\sigma_p>0$ for all primes $p$. This is guaranteed by~\cite[Lemma~7.1]{MR150129}
and the assumption that    $F=0$ is non-singular and has a $\Q$-point.
It is straightforward to see that the sum over $\b t $ forms a multiplicative function of $q$ by using the Chinese Remainder Theorem.
Its value at a prime power $p^k$ is  
$$ 
 \frac{ 1}{ \sigma_p }
\lim_{m\to+\infty } 
\frac{1}{p^{mn}}
\sum_{\substack{\mathbf t \in (\Z/p^k\Z)^{n+1} \\  t_0 \cdots t_n \equiv 0 \md {p^k} } }
\sum_{\substack{\mathbf x \in (\Z/p^m\Z)^{n+1} , \ p^k \mid \b x -\b t  \\ F(\mathbf x)\equiv 0 \md {p^m}   }}
1
 $$ which can be seen to coincide with $h_F(p^k)$ by   interchanging the order of summation.
\end{proof}
The next result will be used to study $h_F$ at prime powers. Its proof is analogous to~\cite[Lemma~3.4]{MR3867326}.

\begin{lemma}
\label{lem:30thousandfeet}
Keep the assumptions of 
Theorem~\ref{thm:spitfire}. Fix any $i\neq j   \in \Z \cap [0,n]$ and $\alpha,\beta \in \mathbb Z\cap [0,\infty)$.
Then there exists $\mu_0>0$ that depends only on $d$ such that 
$$ \lim_{m\to+\infty } \frac{ \#\left\{ \mathbf x \in (\Z/p^m\Z)^{n+1} :p^m\mid F(\mathbf x)  ,p^\alpha \mid x_i, p^\beta \mid x_j \right \}}{p^{mn}}
 = \frac{1}{p^{\alpha+\beta}}(1+O(p^{-1-\mu_0})),$$ where the implied constant depends at most on 
$i, j$ and $F$.
\end{lemma}

\begin{proof}
Let $m\geq\max\{\alpha,\beta\}+1$. Then 
\begin{equation}
\label{eq:kebab}
 \#\left\{ \mathbf x \in (\Z/p^m\Z)^{n+1} :p^m \mid F(\mathbf x)  ,p^\alpha \mid x_i, p^\beta \mid x_j \right \}
=\sum_{\substack{  (x_1,x_2)  \in (\Z/p^m\Z)^{2} \\p^\alpha \mid x_1, p^\beta \mid x_2 }} N(x_1,x_2),\end{equation}
where  $N(x_1,x_2)=\#\{\b y \in (\Z/p^m\Z)^{n-1}: F_{x_1,x_2}(\b y)
\equiv 0 \md{p^m} \}$ 
and    
$$
F_{x_1,x_2}(\b y):= F(y_0,\ldots, y_{i-1}, x_1, y_{i+1},\ldots,y_{j-1}, x_2, y_{j+1},\ldots, y_n ).
$$ 
The expression $$ \frac{1}{p^m} \sum_{a\in \Z/p^m \Z} \exp\left(2 \pi i \frac{a}{p^m} F_{x_1,x_2}(\b y)\right)$$ is $1$ or $0$ according to whether 
$F_{x_1,x_2}(\b y)$ is divisible by $p^m$ or not.
Writing $a=p^{m-t} b$ for some $b \in (\Z/p^t \Z)^*$ the expression becomes $$ \frac{1}{p^m} 
\sum_{t=0}^m
\sum_{b\in (\Z/p^t \Z)^*} \exp\left(2 \pi i \frac{b}{p^t} F_{x_1,x_2}(\b y)\right)=
 \frac{1}{p^m} + \frac{1}{p^m} \sum_{t=1}^m \sum_{b\in (\Z/p^t \Z)^*} \exp\left(2 \pi i \frac{b}{p^t} F_{x_1,x_2}(\b y)\right)
.$$
Hence, $$N(x_1,x_2)=p^{ m(n-2)}
+  \frac{1}{p^m} \sum_{t=1}^m \sum_{b\in (\Z/p^t \Z)^*}
\sum_{\b y \in (\Z/p^m\Z)^{n-1} } 
\exp\left(2 \pi i \frac{b}{p^t} F_{x_1,x_2}(\b y)\right)
.$$ Replacing $\b y$   by its value $\md{p^t}$  does not  affect the exponential, thus, 
$$N(x_1,x_2)=p^{ m(n-2)}
+   p^{m(n-2)}
\sum_{t=1}^m  
p^{-t(n-1)} 
\sum_{b\in (\Z/p^t \Z)^*}  
\sum_{\b z \in (\Z/p^t\Z)^{n-1} } 
\exp\left(2 \pi i \frac{b}{p^t} F_{x_1,x_2}(\b z)\right)
.$$  
We can view $F_{x_1,x_2}(\b y)$ as a  polynomial in $\b y$ since $x_1,x_2$ are fixed. 
It will be non-homogeneous and its degree $d$ part is homogeneous and equals  $F_{0,0}(\b y)$. 
By~\cite[Lemma~3.1]{MR3867326} one has $\mathfrak{B}(F_{x_1,x_2} )\geq \mathfrak{B}(F)-4=n-3$, since 
one must take  $|\b j|_1=2$ and $R=1$ in~\cite{MR3867326}. Our assumptions ensure that $n-3 \geq 1+(d-1)2^d$, thus, 
\cite[Lemma~5.4]{MR150129} shows that,  for every fixed positive $\epsilon$, 
the sum over  
$\b z $ is  
$\ll p^{t(n-1-\mu+\epsilon)}$, where $ \mu= K(F_{x_1,x_2})/(d-1)$
and  $K(g)$ is defined in~\cite[Equation~(8), page 252]{MR150129} as $ \mathfrak B(g^\flat) 2^{-d+1}$.
Note that~\cite[Lemma~5.4]{MR150129} is unaffected by the lower order terms coming from $x_1,x_2$,
since the proof is based on a Weyl differencing process in~\cite[Lemma~2.1]{MR150129}.   
We obtain 
$$ 
N(x_1,x_2)
=p^{ m(n-2)}+ O\Big( p^{ m(n-2)}\sum_{t=1}^m 
p^{t(1-\mu+\epsilon)}  
\Big).$$ We aim to show that $1-\mu<-1$. Using $\mathfrak{B}(F_{x_1,x_2} )\geq n-3$ we obtain 
$$ 
\mu =   \frac{\mathfrak{B}(F_{x_1,x_2} )2^{-d+1} }{d-1}\geq   \frac{(n-3)2^{-d+1} }{d-1} \geq 2+ \frac{2^{-d+1}}{d-1},
$$  
where we used  our assumption $n \geq 4+(d-1)2^d$ in the last step. Let $\epsilon=2^{-d}/(d-1)$ and $\mu_0=-2+\mu-\epsilon$, so that 
 $1-\mu+\epsilon =-1-\mu_0<-1$ and $$ \sum_{t=1}^m p^{t(1-\mu+\epsilon)}
\ll p^{1-\mu+\epsilon}=p^{-1-\mu_0}.$$  Therefore, $
N(x_1,x_2)=p^{ m(n-2)} (1+O(p^{-1-\mu_0}))
$, which can be 
injected into~\eqref{eq:kebab} to obtain 
$$
\#\left\{ \mathbf x \in (\Z/p^m\Z)^{n+1} :p^m \mid F(\mathbf x)  ,p^\alpha \mid x_i, p^\beta \mid x_j \right \}
=p^{mn -\alpha-\beta}  (1+O(p^{-1-\mu_0}))
$$ 
since the right-hand of~\eqref{eq:kebab} has $p^{2m-\alpha-\beta}$ terms. Dividing by $p^{mn}$ concludes the proof.
\end{proof}

\begin{lemma}\label{lem:deathofalgebra}Keep the assumptions of 
Theorem~\ref{thm:spitfire}, let $p$ be a prime and assume that $F=0$ has a $\Q_p$-point. 
There exist positive constants $\delta_F, \delta_F'$ that depend only on $F$ such that for all  $e\geq 1 $ we have 
$$ 
h_F(p) = \frac{n+1}{p} \left(1+O(p^{-1-\delta_F'})\right) \ \ \textrm{ and  } \ \ h_F(p^e) \leq p^{\delta_F-e/(n+1)},
$$
where  the  implied constant depends at most on $F$. 
\end{lemma}

\begin{proof}  
We have 
\begin{align*}
&\frac{ \#\left\{ \mathbf x \in (\Z/p^m\Z)^{n+1} :p^m\mid F(\mathbf x) , p \mid x_0 \cdots x_n   \right \}}{p^{mn}}
\\=\sum_{i=0}^n&\frac{ \#\left\{ \mathbf x \in (\Z/p^m\Z)^{n+1} :p^m\mid F(\mathbf x) , p \mid x_i  \right \}}{p^{mn}}
+O\left(\sum_{1\leq i< j \leq n } \c E_{i,j}\right),\end{align*} where  
$$
\c E_{i,j}
=
\frac{ \#\left\{ \mathbf x \in (\Z/p^m\Z)^{n+1} :p^m\mid F(\mathbf x) , p \mid x_i, p\mid x_j  \right \}}{p^{mn}}
.$$ By Lemma~\ref{lem:30thousandfeet} with $\alpha=\beta=p$ we see that $\c E_{i,j}\ll 1/p^2$.
To estimate the sum over $0\leq i \leq n $ in the main term we use Lemma~\ref{lem:30thousandfeet} with $\alpha=p,\beta=0$
to obtain  $$\sum_{i=0}^n \frac{(1+O(p^{-3/2}))}{p}  +O(n^2  p^{-2} ) =\frac{n+1}{p}+
O(p^{-1-\mu_0} ).$$ Our assumptions ensure that 
$\sigma_p>0$ and recalling~\eqref{eq:renfrewstreet} proves the claimed estimate on $h_F(p)$.

To bound $h_F(p^e)$ note     that if $m>e\geq n$  and $\mathbf x \in (\Z/p^m\Z)^n $ is such that 
$ x_0 \cdots x_n \equiv 0 \md {p^e}$ then  there exists $0\leq i \leq n $ such that $v_p(x_i)\geq e/(n + 1)$.  
By Lemma~\ref{lem:30thousandfeet} 
with $\beta=0$ and $\alpha$ given by the least integer satisfying $\alpha \geq e/(n+1)$ we infer that 
$h_F(p^e) \leq Cp^{-e/(n+1)}$ for some positive constant $C=C(F)$ by taking $m\to\infty$ in the definition of $h_F(p^e)$. 
If $e \in [1,n)$ then we use the trivial bound $h_F(p^e) \leq p^{n+1}$. Thus, in all cases we have shown that 
$$ h_F(p^e) \leq \frac{\max\{C, p^{n+1} \}}{p^{\frac{e}{n+1} }}.$$ Letting $\gamma= (\log C)/(\log 2)$, we infer that 
$C\leq 2^\gamma\leq p^\gamma$, hence, $$ h_F(p^e) \leq p^{\max\{\gamma, n+1\} -e/(n+1) },$$thus, concluding the proof.  \end{proof}
To prove Theorem~\ref{thm:spitfire}  we  can assume with no loss of generality that $F=0$ has a $\Q$-point
so that the set $\mathfrak C:=\{|x_0\cdots x_n|: \b x \in (\Z\setminus\{0\})^{n+1}, F(\b x)=0\}$ is non-empty.
Let $ \mathcal A =\{\b x \in (\Z\setminus\{0\})^{n+1}: F(\b x)=0\}$ and for every $a=\b x \in \mathcal A$ define $c_a= |x_0\cdots x_n|$.
Define  $\chi_B:\mathcal A\to [0,\infty)$  by 
$$\chi_B(\b x):=\mathds 1_{[0,B]}(\max\{|x_i|:0\leq i \leq n \}).$$ By 
Lemma~\ref{lem:birlevel}   there exist $\Xi,\beta>0$ such that for all  $q\leq B^\Xi$ one has 
$$\sum_{\substack{ a\in \mathcal A  \\ q\mid c_a}} \chi_B(a)=\sigma(F) B^{n+1-d} (h_F(q) +O(B^{-\beta} ) ).$$ 
This shows that  the property in Definition~\ref{def:levdistr} is fulfilled   for $M= \sigma(F) B^{n+1-d}$. In particular, 
$$\sup\{c_a:\chi_B(a)>0\}\leq \sup\{|x_0\cdots x_n|: \max|x_i| \leq B \} 
\leq B^{n+1}\ll M^{(n+1)/(n+1-d)},$$ thus,~\eqref{Roll Over Beethoven} holds with $\alpha= \frac{n+1}{n+1-d}$.
Note that $h_F\in \mathcal D(n+1,  1/(n+1), \delta_F, B', K)$ for some positive constants $B',K$ that depend only on $F$ due to $h_F(p^e)\leq h_F(p)$ and Lemma~\ref{lem:deathofalgebra}.
We can thus employ Theorem~\ref{tMain}  to deduce  that $$\sum_{\substack{ \b x \in (\Z\setminus\{0\})^{n+1}\\  \max |x_i|\leq B, F(\b x)=0  }}
\hspace{-0,5cm}f(|x_0 \cdots x_n |) \ll  B^{n+1-d} \prod_{\substack{ p\leq M    } } (1-h_F(p ) ) 
\sum_{\substack{ a\leq M }} f(a)     h_F(a),
$$
where $M=\sigma(F)B^{n+1-d}$. By Lemma~\ref{lem:deathofalgebra} we can bound the product over $p\leq M $ by 
$\ll \exp(-(n+1)\sum_{p\leq M} 1/p)$. Combining this with the succeding lemma completes the proof of Theorem~\ref{thm:spitfire}. 
\begin{lemma}
\label{lem:bwv830}
Keep the setting of Theorem~\ref{thm:spitfire}. 
 For every fixed constant $\gamma>0$ and each $B\geq 1 $ 
we have 
$$
\sum_{\substack{ a\leq B^\gamma }} f(a)     h_F(a) \ll  \exp \left((n+1)  \sum_{p\leq B } \frac{f(p)}{p}\right),
$$ where the implied constant depends at most on $f,A,\gamma$ and $F$.
\end{lemma}
\begin{proof}
We will apply Lemma~\ref{lem:lord of all fevers and plague} with $G=f$.
It is clear that $f$ satisfies the required assumptions. 
We next verify the required assumptions for   $h_F$: the bound~\eqref{eFbound} holds for $h_F(p^e)$ due to 
Lemma~\ref{lem:deathofalgebra} and the fact that $h_F(p^e) \leq h_F(p)$.
It remains to prove the estimate
 $h_F(p^e) \ll 1/p^2$ 
 for all $e\geq 2$.
Since $h_F(p^e)\leq h_F(p^2)$ it suffices to bound  $h_F(p^2)$. To do so we  
 note that if $\b x \in (\Z/p^m\Z)^{n+1}$ is such that $F(\b x ) \equiv 0 \md{p^m}$ and $p^2 \mid x_0\cdots x_n$, then 
either there exists $i$ such that $p^2 \mid x_i$ or there are $i\neq j $ such that $p\mid x_i$ and $p\mid x_j$.  In the first case we may employ 
Lemma~\ref{lem:30thousandfeet} with $\alpha=2$ and $\beta=0$ and in the second case with $\alpha=1$ and $\beta=1$ in order to obtain the bound 
$h_F(p^2)\ll p^{-2}$.  Now
 Lemma~\ref{lem:lord of all fevers and plague} implies that 
\[\sum_{a\leq B^{\gamma}}f(a)h_F(a)\ll \exp\left(\sum_{p\leq B^{\gamma}}f(p)h_F(p)\right),\]
where the implied constant only depends on $f$ and $F$. Using the estimate for  $h_F(p)$
from Lemma~\ref{lem:deathofalgebra} and  noting that $f(p)\leq A$ and $\sum p^{-2}\leq \infty$, we can rewrite the sum over $p\leq B^{\gamma}$ as
\[\sum_{p\leq B^{\gamma}}f(p)h_F(p)=(n+1)\sum_{p\leq B^{\gamma}}\frac{f(p)}{p}+O\left((n+1)A\right).\]
Therefore
\[\sum_{a\leq B^{\gamma}}f(a)h_F(a)\ll \exp\left((n+1)\sum_{p\leq B^{\gamma}}\frac{f(p)}{p}\right),\]
where the implied constant depends at most on $f$, $A$, $n$ and $h_F$.
 To conclude the proof we note that for all $0<\gamma_1<\gamma_2$ one has 
$$\sum_{B^{\gamma_1} < p\leq B^{\gamma_2} } \frac{f(p)}{p} \leq A \sum_{B^{\gamma_1} < p\leq B^{\gamma_2} } \frac{1}{p}=A\log \frac{\gamma_2}{\gamma_1}+O(1/\log B)
\ll 1,$$thus, one can replace the condition  $ p\leq B^\gamma$ by $p\leq B$ at the cost of a different implied constant. 
\end{proof}

\section{Polynomial values}\label{s:nrtnbm}
In this section we give upper and lower bounds for sums of the form~\eqref{eq:sumsrrr}.
\subsection{Proving equidistribution}  Here we prove the necessary   results that will be fed into 
Theorems~\ref{tMain}-\ref{thm:lower} to yield  Theorems~\ref{thm:nrtnbm}-\ref{thm:lowernair}.

If  $p_Q$ denotes the largest prime dividing all coefficients of $Q$ then for $p\leq p_Q$ we 
have $\varrho_Q(p^e) \leq 1 \leq p_Q/p$ for all $e\geq 1 $. For all other primes we   use~\cite[Lemma~2.7]{MR3988669}   to obtain  
\begin{equation}
\label{eq:wang} 
\varrho_Q(p^e) \leq \frac{p_Q+\deg(Q)}{p}.
\end{equation} 
Furthermore, Lemma~\ref{lem:wud} gives  
\begin{equation}
\label{eq:wang2}
\varrho_Q(p^e) \leq \frac{C_0 }{p^{e/\deg(Q)}}
\end{equation}
for a positive constant $C_0$ that only depends 
on $Q$. Finally, Lemma~\ref{lem:chebotrv} shows   
\begin{equation}\label{eq:wang3}\prod_{p\leq x}  (1- \varrho_Q(q))
=c (\log x)^{-r} (1+O(1/\log x)),\end{equation} where $r$ is the number of irreducible components of $Q$ and $c$ is a positive number depending only on $Q$.

The following definition is based on~\cite[Definition 2.2]{MR2846337}. Write $|\cdot|$ for the usual Euclidean norm on $\mathbb{R}^k$ for $k\in \mathbb N$.

\begin{definition}[Regions]
Let $n, M \geq 1$ be integers and let $L > 0$ be a real number. We say that a subset $S$ of $\mathbb{R}^n$ is in $\mathrm{Reg}(n, M, L)$ if
\begin{enumerate}
\item $S$ is bounded, 
\item there exist $M$ maps $\phi_1, \dots, \phi_M: [0, 1]^{n - 1} \rightarrow \mathbb{R}^n$ satisfying
\[
|\phi_i(\mathbf{x}) - \phi_i(\mathbf{y})| \leq L |\mathbf{x} - \mathbf{y}|
\]
for $\mathbf{x}, \mathbf{y} \in [0, 1]^{n - 1}$ and $i = 1, \dots, M$ such that the images of $\phi_i$ cover $\partial S$.
\end{enumerate}
\end{definition}

\begin{lemma}[Widmer, {\cite[Theorem~2.4]{MR2846337}}]
\label{tWidmer} Each 
  $S \in \mathrm{Reg}(n, M, L)$ is  measurable and  satisfies 
\[
\left|\#(S \cap \mathbb{Z}^n)- \mathrm{vol}(S)\right| \leq n^{\frac{3n^2}{2}} M \max(L^{n - 1}, 1).
\]
\end{lemma}

\begin{lemma}\label{lem:lips}
Let $\mathcal D \in \mathrm{Reg}(n, M, L)$. Then for all  $\mathbf w \in [-1,1]^n$  and  $t\geq 1 $ one has 
$$\#\{\mathbf x\in \mathbb Z^n \cap (t \mathcal D+\mathbf w) \}=\mathrm{vol}(\mathcal D) t^n +O(t^{n-1}),
$$ where the implied constant depends on $n, M, L$, but is independent of $t$ and  $\mathbf w$.
\end{lemma}\begin{proof}This follows from Lemma~\ref{tWidmer}. \end{proof}

\begin{lemma}
\label{lem:polynomresiduesclasses} 
Let $\mathcal B \in \mathrm{Reg}(n, M, L)$ and let $Q\in \Z[x_1,\ldots, x_n]$ be non-zero.
Then for all $\b c \in \R^n, q\in \mathbb N, R\geq 1 $ with $q\leq R $  we have 
$$\#\{   \b x \in \Z ^{n }\cap (R \mathcal B + \b c)  :  Q(\b x)\neq 0 ,  q \mid Q(\b x) \}
=\varrho_Q(q)\mathrm{vol}(\mathcal B) R^n (1+O(q/R)  ) +O(R^{n-1}),$$ 
where  the implied constant   depends     on $n$, $\mathcal B$ and $Q$,
but is independent of $R$ and $q$. 
\end{lemma} 

\begin{proof} 
We may assume that $Q$ is non-constant. We will first show that
\[
\# \{\b x \in \Z^n \cap  (R \mathcal B + \b c) : Q(\b x) = 0\} \ll R^{n - 1},
\]
where  the implied constant depends at most on $n$, $\mathcal B$ and $Q$. We will proceed by induction on $n$. The case $n = 1$ is easy, so now suppose that $n > 1$. Observe that we may certainly assume that the variable $x_n$ occurs in $Q$. Expand $Q$ as
\[
Q(x_1, \dots, x_n) = \sum_{i = 0}^m Q_i(x_1, \dots, x_{n - 1}) x_n^i
\]
with $m \geq 1$ and $Q_m(x_1, \dots, x_{n - 1})$ non-zero. By the induction hypothesis we may bound the number of zeros of $Q_m(x_1, \dots, x_{n - 1})$. If $Q_m(x_1, \dots, x_{n - 1}) \neq 0$, then there are at most $m$ possibilities, say $\alpha_1, \dots, \alpha_m$, for $x_n$. We may find $B$ such that $\mathcal{B}$ is contained in $[-B, B]^n$. Then we can employ Lemma~\ref{lem:lips} in     dimension $n-1$ to see that the number of integer vectors in $\b x \in \Z^n \cap  (R[-B, B]^n + \b c) $ for which $x_n=\alpha_j$ is $\ll R^{n-1}$, 
where  the implied constant   depends at most on $n$ and $\mathcal B$. We can thus add back the missing terms 
with $Q(\b x )=0 $ at the cost of a negligible error term depending only on $n$, $\mathcal{B}$ and $Q$.

Note that for every $c\in \mathbb R$ and $y' \in \mathbb Z/q\Z$ there exists a unique $y\in \mathbb Z \cap [c,c+q)$ such that $y\equiv y' \md q$.
Using this for every $y_i$ and $c_i$ allows us to 
to split  in progressions in order to obtain 
\begin{equation}\label{eq:jacklousier}\#\{   \b x \in \Z^n\cap  (R \mathcal B + \b c)  :    q \mid Q(\b x) \}
=\sum_{\substack{\b y \in  \Z^n \cap \mathcal I \\ q\mid Q(\b y )   } }\#\{   \b x \in \Z^n\cap  (R \mathcal B + \b c) :   \b x \equiv \b y \md q \},
\end{equation} where $\mathcal I$ is the product of intervals $\prod_{i=1}^n [c_i,c_i+q)$.
Letting $\mathbf x = \mathbf y+q\mathbf z $ we infer   
$$\#\{   \b x \in \Z^n\cap  (R \mathcal B + \b c) :   \b x \equiv \b y \md q \} =
\#\left \{\mathbf z\in \mathbb Z^n\cap  \Big( \frac{R}{q} \mathcal B + \frac{\b c-\mathbf y}{q}  \Big) \right \} .$$ 
Since $y_i\in [c_i,c_i+q)$  and $R/q \geq 1 $ we can use Lemma~\ref{lem:lips} with $\mathbf w =(\b c -\b y )/q$ to obtain 
$$ 
\#\left \{\mathbf z\in \mathbb Z^n\cap  \Big( \frac{R}{q} \mathcal B + \frac{\b c-\mathbf y}{q}  \Big) \right \} =
\mathrm{vol}(\mathcal B) \frac{R^n}{q^n}+O\left( \frac{R^{n-1} }{q^{n-1} }\right)  ,
$$
where the implied constant depends at most on $n$ and $\mathcal{B}$. Injecting this into~\eqref{eq:jacklousier} gives 
 $$\#\{   \b x \in \Z^n\cap  (R \mathcal B + \b c) :    q \mid Q(\b x) \}
=\varrho_Q(q)(\mathrm{vol}(\mathcal B) R^n +O(q R^{n-1})  )
,$$ which completes the proof.
 \end{proof}
\subsection{Upper bounds}\label{ss:gnrluper}
We have $[-1,1]^n \in \mathrm{Reg}(n, 2^n, 2)$. Thus, by Lemma~\ref{lem:polynomresiduesclasses} with $\mathcal B=[-1,1]^n$ and 
$\b c =\b 0 $ we have 
$$ \#\{   \b x \in \Z ^{n }\cap T[-1,1]^n :  Q(\b x)\neq 0 ,  q \mid Q(\b x) \}
= \varrho_Q(q)2^n   T^n\big(1+O\big(q/T\big)\big)+O(  T^{n-1}),$$
whenever $q\leq T$, where the implied constant depends at most on $n,Q$ and $\mathcal B$. 
In particular, if $q\leq T^{1/2} $ then the right-hand side becomes 
 $$ 
\varrho_Q(q)2^n  T^n(1+O(1/\sqrt T))+O(  T^{n-1/2}).$$
Thus, we can  use  Theorem~\ref{tMain} with 
$$\mathcal A=\{\b x \in \Z^n: Q(\b x )\neq 0 \}, \mathfrak C=\{|Q(\b x )|:\b x \in \mathcal A\}, \chi_T(\b x)=\mathds 1_{[0,T]} (\max|x_i|),$$
since it shows that the assumption of Definition~\ref{def:levdistr}  holds with $$h=\varrho_Q, M(T)=2^n  T^n, \theta=\frac{1}{4n}, \xi=\frac{1}{2n}.$$
Note that   assumption~\eqref{elowaverg} is satisfied with $\kappa=r$ due to~\eqref{eq:wang3}, while 
assumptions~\eqref{eLowPower}-\eqref{eHighPower} hold respectively due to~\eqref{eq:wang}-\eqref{eq:wang2}. Hence for all $T\geq 1 $  we have
$$\sum_{\substack{ \b x \in \Z^n \cap  T [-1,1]^n\\ Q(\b x )\neq 0  } } f(|Q(\b x ) |)
  \ll  T^n \prod_{\substack{  p\leq T^n    } } (1-\varrho_Q(p) ) \sum_{\substack{ a\leq T^n  }} f(a)  \varrho_Q(a),$$
where the implied constant depends at most on $A, Q$ and $n$.   
By~\eqref{eq:wang3} the product over $p\leq T^n$ is asymptotic to $(\log T)^{-r}$.

To complete the proof of Theorem~\ref{thm:nrtnbm} let us note that 
since $f\geq 0 $ and $\mathcal D$ is contained in $ X(\mathcal D) [-1,1]^n$ we can write 
$$ \sum_{\substack{ \mathbf x \in  \mathbb Z^n \cap  \mathcal  D \\ Q(\b x) \neq 0   }} f(|Q(\b x) |)\leq 
\sum_{\substack{ \mathbf x \in  \mathbb Z^n \cap X(\mathcal D)[-1,1]^n \\ Q(\b x) \neq 0   }} f(|Q(\b x) |).$$ 
As we have seen, this is   $$ \ll  X(\mathcal D)^n 
( \log 2X(\mathcal D))^{-r}   \sum_{\substack{ a\leq X(\mathcal D) ^n}} f(a)   
\varrho_Q(a),$$ which is sufficient.

\subsection{Lower bounds}\label{ss:gnrlowerrs}
By assumption $\mathcal D$ contains a set of the form $\b c+ X \mathcal U$, where $\mathcal U$ is the 
unit ball in $\R^n$. Since $f\geq 0 $ we deduce that 
$$ \sum_{\substack{ \mathbf x \in  \mathbb Z^n \cap  \mathcal  D \\ Q(\b x) \neq 0   }} f(|Q(\b x) |)\geq 
\sum_{\substack{ \mathbf x \in  \mathbb Z^n \cap (\b c+ X \mathcal U ) \\ Q(\b x) \neq 0   }} f(|Q(\b x) |).$$ 
We may thus employ Theorem~\ref{thm:lower} with 
$$\mathcal A=\{\b x \in \Z^n: Q(\b x )\neq 0 \}, \mathfrak C=\{|Q(\b x )|:\b x \in \mathcal A\}, \chi_t(\b x)=\mathds 1_{\b c+ X \mathcal U} (\b x).$$
Note that $\mathcal U $ can be parametrised by a single function in $n-1$ variables by using trigonometric functions, thus, 
we can use Lemma~\ref{lem:polynomresiduesclasses} to obtain 
 $$ \#\{   \b x \in \Z ^{n }\cap \b c+ X \mathcal U :  Q(\b x)\neq 0 ,  q \mid Q(\b x) \}
= \mathrm{vol}(\mathcal U)  \varrho_Q(q) X^n+O(  \deg(Q)^{\omega(q)}  X^{n-1}),$$
whenever $q\leq X$, where the implied constant depends at most on $n$ and $Q$.
The proof can now be completed by following the same steps as in the final stage of the proof of Theorem~\ref{thm:lowernair}.

\subsection{Proof of Theorem~\ref{cor:nrtnbmadivisor}}
\label{s:prflastcorol}
When $\mathcal D=[-X,X]^n$, the parameter $X(\mathcal D)$ in Definition~\ref{def:se pente lepta} satisfies $X(\mathcal D)\leq 2X$. Thus, 
Theorem~\ref{thm:nrtnbm} yields that  
$$\sum_{\substack{ \mathbf x \in  \mathbb Z^n \cap [-X,X]^n  \\ Q(\b x) \neq 0   }} \tau_k(|Q(\b x) |)^\ell  \ll 
X^n
(\log X) ^{-1} 
\sum_{\substack{ a\leq X^n}}  \tau_k(a)^\ell  \varrho_Q(a)
.$$ A   lower bound can be given by
employing Theorem~\ref{thm:lowernair} with $\mathcal D=[-X,X]^n$. This is allowed since 
  the sphere in $\mathbb R^n$ with radius $X$   and centre at the origin
is contained in $\mathcal D$. To conclude the proof we use Lemma~\ref{lem:buchanan bus station}
with $G=\tau_k^\ell$. To see that $G$ satisfies the required assumptions 
we recall that for   primes $p$ and all $e\geq 1 $ one has 
$$\tau_k(p^e) = \frac{(e+k-1)(e+k-2)\cdots (e+1) }{(k-1)!}
$$ so that $\tau_k(p)^\ell =k^\ell $ and $\tau_k(p^e)^\ell \leq e^{O_{k,\ell}(1)}
  \ll_{k,\ell} (3/2)^{e}$.
Hence, Lemma~\ref{lem:buchanan bus station} can be employed with $r=1, \lambda=k^\ell$ and $C=3/2$; it yields 
$$\sum_{\substack{ a\leq X}}  \tau_k(a)^\ell  \varrho_Q(a) \asymp (\log X)^{k^\ell  }.$$

\end{document}